\documentclass[11pt]{article}
\usepackage{fullpage}

\usepackage[utf8]{inputenc} 
\usepackage[T1]{fontenc}    
\usepackage{hyperref}       
\usepackage{url}            
\usepackage{booktabs}       
\usepackage{amsfonts}       
\usepackage{nicefrac}       
\usepackage{microtype}      

\usepackage{amsmath}
\usepackage{amssymb}
\usepackage{xcolor}
\usepackage{subfigure}

\newcommand{\comm}{{\rm comm}}
\newcommand{\comp}{{\rm comp}}
\newcommand{\rel}{{\rm rel}}

\newcommand{\esp}[1]{\mathbb{E}\left[#1\right]}
\newcommand{\espk}[2]{\mathbb{E}_{#1}\left[#2\right]}
\newcommand{\R}{\mathbb{R}}
\newcommand{\cB}{\mathcal{B}}
\newcommand{\cN}{\mathcal{N}}
\newcommand{\one}{\mathds{1}}
\newcommand{\Ker}{{\rm Ker}}
\newcommand{\cL}{\mathcal{L}}

\newcommand{\out}{{\rm out}}

\usepackage{amsthm}
\newtheorem{assumption}{Assumption}[]
\newtheorem{theorem}{Theorem}[]
\newtheorem{lemma}{Lemma}[]

\newtheorem{corollary}{Corollary}[]

\usepackage{graphicx}
\usepackage{algorithm}
\usepackage[noend]{algorithmic}

\usepackage{natbib}
\usepackage{dsfont}

\title{Dual-Free Stochastic Decentralized Optimization\\
with Variance Reduction}

\author{
Hadrien Hendrikx\thanks{D\'epartement d’informatique de l’ENS, ENS, CNRS, PSL University, Paris, France} \footnotemark[2]\\
(\texttt{hadrien.hendrikx@inria.fr})
\and
Francis Bach\footnotemark[1] \thanks{INRIA, Paris, France}\\
(\texttt{francis.bach@inria.fr}) 
\and
Laurent Massouli\'e\footnotemark[1] \footnotemark[2]\\
(\texttt{laurent.massoulie@inria.fr})
}

\date{}

\begin{document}

\maketitle

\begin{abstract}
    We consider the problem of training machine learning models on distributed data in a decentralized way. For finite-sum problems, fast single-machine algorithms for large datasets rely on stochastic updates combined with variance reduction. Yet, existing decentralized stochastic algorithms either do not obtain the full speedup allowed by stochastic updates, or require oracles that are more expensive than regular gradients. In this work, we introduce a Decentralized stochastic algorithm with Variance Reduction called DVR. DVR only requires computing stochastic gradients of the local functions, and is computationally as fast as a standard stochastic variance-reduced algorithms run on a $1/n$ fraction of the dataset, where $n$ is the number of nodes. To derive DVR, we use Bregman coordinate descent on a well-chosen dual problem, and obtain a dual-free algorithm using a specific Bregman divergence. We give an accelerated version of DVR based on the Catalyst framework, and illustrate its effectiveness with simulations on real data.
\end{abstract}

\section{Introduction}
\label{sec:introduction}
We consider the regularized empirical risk minimization problem distributed on a network of $n$ nodes. Each node has a local dataset of size $m$, and the problem thus writes:\vspace{-3pt}
\begin{equation}\label{eq:general_problem}
    \min_{x \in \R^d} F(x) \triangleq \sum_{i=1}^n f_i(x), \hbox{ with } f_i(x) \triangleq \frac{\sigma_i}{2} \|x\|^2 + \sum_{j=1}^m f_{ij}(x),
\end{equation}
where $f_{ij}$ typically corresponds to the loss function for training example $j$ of machine $i$, and $\sigma_i$ is the local regularization parameter for node $i$. We assume that each function $f_{ij}$ is convex and $L_{ij}$-smooth (see, \emph{e.g.},~\citet{nesterov2013introductory}), and that each function $f_i$ is $M_i$-smooth. Following~\citet{xiao2017dscovr}, we denote $\kappa_i = (1 + \sum_{i=1}^m L_{ij}) / \sigma_i$ the stochastic condition number of $f_i$, and $\kappa_s = \max_i \kappa_i$. Similarly, the batch condition number is $\kappa_b = \max_i M_i / \sigma_i$. It always holds that $\kappa_b \leq \kappa_s \leq m \kappa_b$, but generally $\kappa_s \ll m \kappa_b$, which explains the success of stochastic methods. Indeed, $\kappa_s \approx m\kappa_b$ when all Hessians are orthogonal to one another which is rarely the case in practice, especially for a large dataset.

Regarding the distributed aspect, we follow the standard \emph{gossip} framework~\citep{boyd2006randomized, nedic2009distributed, duchi2012dual, scaman2017optimal} and assume that nodes are linked by a communication network which we represent as an undirected graph $G$. We denote $\cN(i)$ the set of neighbors of node $i$ and $\one \in \R^d$ the vector with all coordinates equal to $1$. Communication is abstracted by multiplication by a positive semi-definite matrix $W \in \R^{n \times n}$, which is such that $W_{k\ell} = 0$ if $k \notin \cN(\ell)$, and $\Ker(W) = {\rm Span}(\one)$. The matrix $W$ is called the \emph{gossip matrix}, and we denote its spectral gap by $\gamma = \lambda_{\min}^+(W) / \lambda_{\max}(W)$, the ratio between the smallest non-zero and the highest eigenvalue of $W$, which is a key quantity in decentralized optimization. We finally assume that nodes can compute a local stochastic gradient $\nabla f_{ij}$ in time $1$, and that communication (\emph{i.e.}, multiplication by $W$) takes time $\tau$. 

\textbf{Single-machine stochastic methods.} Problem~\eqref{eq:general_problem} is generally solved using first-order methods. When $m$ is large, computing $\nabla F$ becomes very expensive, and batch methods require $O(\kappa_b\log(\varepsilon^{-1}))$ iterations, which takes time $O(m\kappa_b\log(\varepsilon^{-1}))$, to minimize $F$ up to precision $\varepsilon$. In this case, updates using the stochastic gradients $\nabla f_{ij}$, where $(i,j)$ is selected randomly, can be much more effective~\citep{bottou2010large}. Yet, these updates are noisy and plain stochastic gradient descent (SGD) does not converge to the exact solution unless the step-size goes to zero, which slows down the algorithm. One way to fix this problem is to use variance-reduced methods such as SAG~\citep{schmidt2017minimizing}, SDCA~\citep{shalev2013stochastic}, SVRG~\citep{johnson2013accelerating} or SAGA~\citep{defazio2014saga}. These methods require $O((nm + \kappa_s)\log(\varepsilon^{-1}))$ stochastic gradient evaluations, which can be much smaller than $O(m\kappa_b\log(\varepsilon^{-1}))$. 

\textbf{Decentralized methods.} Decentralized adaptations of gradient descent in the smooth and strongly convex setting include EXTRA~\citep{shi2015extra}, DIGing~\citep{nedic2017achieving} or NIDS~\citep{li2019decentralized}. These algorithms have sparked a lot of interest, and the latest convergence results~\citep{jakovetic2018unification, xu2020distributed, li2020revisiting} show that EXTRA and NIDS require time $O((\kappa_b + \gamma^{-1})(m + \tau))\log(\varepsilon^{-1}))$ to reach precision $\varepsilon$. A generic acceleration of EXTRA using Catalyst~\citep{li2020revisiting} obtains the (batch) optimal $O(\sqrt{\kappa_b}(1 + \tau / \sqrt{\gamma})\log(\varepsilon^{-1})) $ rate up to log factors. Another line of work on decentralized algorithms is based on the \emph{penalty method}~\citep{li2018sharp, dvinskikh2019decentralized}. This consists in performing traditional optimization algorithms to problems augmented with a Laplacian penalty, and in particular enables the use of accelerated methods. Yet, these algorithms are sensitive to the value of the penalty parameter (when it is fixed), since it directly influences the solution they converge to. Another natural way to construct decentralized optimization algorithms is through dual approaches~\citep{scaman2017optimal, uribe2020dual}. Although the dual approach leads to algorithms that are optimal both in terms of number of communications and computations~\citep{scaman2019optimal, hendrikx2020optimal}, they generally assume access to the proximal operator or the gradient of the Fenchel conjugate of the local functions, which is not very practical in general since it requires solving a subproblem at each step. 

\textbf{Decentralized stochastic optimization.} Although both stochastic and decentralized methods have a rich litterature, there exist few decentralized stochastic methods with linear convergence rate. Although DSA~\citep{mokhtari2016dsa}, or GT-SAGA~\citep{xin2020decentralized} propose such algorithms, they respectively take time $O((m \kappa_s + \kappa_s^4 \gamma^{-1} (1 + \tau) \log(\varepsilon^{-1}))$ and $O((m + \kappa_s^2 \gamma^{-2})(1 + \tau) \log(\varepsilon^{-1}))$ to reach precision $\varepsilon$. Therefore, they have significantly worse rates than decentralized batch methods when $m=1$, and than single-machine stochastic methods when $n=1$. Other methods have better rates of convergence~\citep{shen2018towards, hendrikx2019accelerated} but they require evaluation of proximal operators, which may be expensive.

\textbf{Our contributions.} This work develops a dual approach similar to that of~\citet{hendrikx2019accelerated}, which leads to a decentralized stochastic algorithm with rate $O(m + \kappa_s + \tau \kappa_b / \sqrt{\gamma})$, where the $\sqrt{\gamma}$ factor comes from Chebyshev acceleration, such as used in~\citet{scaman2017optimal}. Yet, our algorithm, called DVR, can be formulated in the primal only, thus avoiding the need for computing expensive dual gradients or proximal operators. Besides, DVR is derived by applying Bregman coordinate descent to the dual of a specific augmented problem. Thus, its convergence follows from the convergence of block coordinate descent with Bregman gradients, which we prove as a side contribution. When executed on a single-machine, DVR is similar to dual-free SDCA~\citep{shalev2016sdca}, and obtains similar rates. We believe that the same methodology could be applied to tackle non-convex problems, but we leave these extensions for future work.

We present in Section~\ref{sec:algo_design} the derivations leading to DVR, namely the dual approach and the dual-free trick. Then, Section~\ref{sec:algo_results} presents the actual algorithm along with a convergence theorem based on block Bregman coordinate descent (presented in Appendix~\ref{app:bregman_coordinate_gradient}). Section~\ref{sec:acceleration} shows how to accelerate DVR, both in terms of network dependence (Chebyshev acceleration) and global iteration complexity (Catalyst acceleration~\citep{lin2017catalyst}).
Finally, experiments on real-world data are presented in Section~\ref{sec:experiments}, that demonstrate the effectiveness of DVR. 

\section{Algorithm Design} \label{sec:algo_design}
This section presents the key steps leading to DVR. We start by introducing a relevant dual formulation from~\citet{hendrikx2019accelerated}, then introduce the dual-free trick based on~\citet{lan2017optimal}, and finally show how this leads to DVR, an actual implementable decentralized stochastic algorithm, as a special case of the previous derivations.

\subsection{Dual formulation}

The standard dual formulation of Problem~\eqref{eq:general_problem} is obtained by associating a parameter vector to each node, and imposing that two neighboring nodes have the same parameters~\citep{boyd2011distributed, jakovetic2014linear, scaman2017optimal}. This leads to the following constrained problem, in which we write $\theta^{(i)} \in \R^d$ the local vector of node $i$: 
\begin{equation} \label{eq:constrained_standard}
\min_{\theta \in \R^{nd}} \sum_{i=1}^n f_i(\theta^{(i)}) \text{ such that } \forall k, \ell \in \cN(k), \ \theta^{(k)} = \theta^{(\ell)}.
\end{equation}
Following the approach of~\citet{hendrikx2019accelerated, hendrikx2020optimal}, we further split the $f_i(\theta^{(i)})$ term into $\sigma_i\|\theta^{(i)}\|^2/2 + \sum_{j=1}^n f_{ij}(\theta^{(ij)})$, with the constraint that $\theta^{(i)} = \theta^{(ij)}$ for all $j$. This is equivalent to the previous approach performed on an augmented graph~\citep{hendrikx2019accelerated, hendrikx2020optimal} in which each node is split into a star network with the regularization in the center and a local summand at each tip of the star. Thus, the equivalent augmented constrained problem that we consider writes:
\begin{equation} \label{eq:constrained_augmented}
\!\!\min_{\theta \in \R^{n(m+1)d}} \sum_{i=1}^n \! \left[\frac{\sigma_i}{2}\|\theta^{(i)}\|^2\! + \! \sum_{j=1}^m f_{ij}(\theta^{(ij)})\right] \text{s.t. } \forall k, \ell \in \cN(k),\ \theta^{(k)} = \theta^{(\ell)} \text{ and }  \forall i,j, \ \theta^{(i)} = \theta^{(ij)}.
\end{equation}
We now use Lagrangian duality, and introduce two kinds of multipliers. The variable $x$ corresponds to multipliers associated with the constraints given by edges of the communication graph (\emph{i.e.}, $\theta^{(k)} = \theta^{(\ell)} \text{ if } k \in \cN(\ell)$), that we will call \emph{communication edges}. Similarly, $y$ corresponds to the constraints associated with the edges that are specific to the augmented graph (\emph{i.e.}, $\theta^{(i)} = \theta^{(ij)} \ \forall i,j$) that we call \emph{computation} or \emph{virtual edges}, since they are not present in the original graph and were constructed for the augmented problem. Therefore, there are $E$ communication edges (number of edges in the initial graph), and $nm$ virtual edges. The dual formulation of Problem~\eqref{eq:constrained_augmented} thus writes:
\begin{equation} \label{eq:dual_problem}
    \min_{x \in \R^{Ed},\  y \in \R^{nmd}} \frac{1}{2} q_A(x,y) + \sum_{i=1}^n \sum_{j=1}^m f_{ij}^*((A(x,y))^{(ij)}), \hbox{ with } q_A(x,y) \triangleq (x, y)^\top A^\top \Sigma  A (x, y),
\end{equation}
and where $(x, y) \in \R^{(E + nm)d}$ is the concatenation of vectors $x \in \R^{Ed}$, which is associated with the communication edges, and $y \in \R^{n md}$, which is the vector associated with computation edges. We denote $\Sigma = {\rm Diag}(\sigma_1^{-1}, \cdots, \sigma_n^{-1}, 0, \cdots, 0) \otimes I_d \in \R^{n(m+1)d \times n(m+1)d}$ and $A$ is such that for all $z \in \R^d$, $A (e_{k,\ell} \otimes z) = \mu_{k\ell}(u_k - u_\ell) \otimes P_{k\ell} z$ for edge $(k,\ell)$, where $P_{k\ell} = I_d$ if $(k, \ell)$ is a communication edge, $P_{ij}$ is the projector on $\Ker(f_{ij})^\perp \triangleq (\cap_{x \in \R^d} \Ker(\nabla^2 f_{ij}(x)))^\perp$ if $(i,j)$ is a virtual edge, $z_1 \otimes z_2$ is the Kronecker product of vectors $z_1$ and $z_2$, and $e_{k,\ell} \in \R^{E + nm}$ and $u_k \in R^{n(m+1)}$ are the unit vectors associated with edge $(k, \ell)$ and node $k$ respectively.

Note that the upper left $nd \times nd$ block of $AA^\top$ (corresponding to the communication edges) is equal to $W \otimes I_d$ where $W$ is a gossip matrix (see, \emph{e.g.},~\citep{scaman2017optimal}) that depends on the $\mu_{k\ell}$. In particular, $W$ is equal to the Laplacian of the communication graph if $\mu_{k\ell}^2 = 1/2$ for all $(k,\ell)$. For computation edges, the projectors $P_{ij}$ account for the fact that the parameters $\theta^{(i)}$ and $\theta^{(ij)}$ only need to be equal on the subspaces on which $f_{ij}$ is not constant, and we choose $\mu_{ij}$ such that $\mu_{ij}^2 = \alpha L_{ij}$ for some $\alpha > 0$. Although this introduces heavier notations, explicitly writing $A$ as an $n(1+m)d \times (E + nm)d$ matrix instead of an $n(1+m) \times (E + nm)$ matrix allows to introduce the projectors $P_{ij}$, which then yields a better communication complexity than choosing $P_{ij} = I_d$. See~\citet{hendrikx2019accelerated, hendrikx2020optimal} for more details on this dual formulation, and in particular on the construction on the augmented graph. Now that we have obtained a suitable dual problem, we would like to solve it without computing gradients or proximal operators of $f_{ij}^*$, which can be very expensive.

\subsection{Dual-free trick}
Dual methods are based on variants of Problem~\eqref{eq:dual_problem}, and apply different algorithms to it. In particular, \citet{scaman2017optimal, uribe2020dual} use accelerated gradient descent~\citep{nesterov2013introductory}, and~\citet{hendrikx2018accelerated, hendrikx2019accelerated} use accelerated (proximal) coordinate descent~\citep{lin2015accelerated}. Let $p_\comm$ denote the probability of performing a communication step and $p_{ij}$ be the probability that node $i$ samples a gradient of $f_{ij}$, which are such that for all $i$, $\sum_{j=1}^m p_{ij} = 1 - p_\comm$.  Applying a coordinate update with step-size $\eta / p_\comm$ to Problem~\eqref{eq:dual_problem} in the direction $x$ (associated with communication edges) writes:
\begin{equation}
    x_{t+1} = x_t - \eta p_\comm^{-1} \nabla_x q_A(x_t, y_t),
\end{equation}
where we denote $\nabla_x$ the gradient in coordinates that correspond to $x$ (communication edges), and $\nabla_{y, ij}$ the gradient for coordinate $(ij)$ (computation edge). Similarly, the standard coordinate update of a local computation edge $(i,j)$ can be written as:
\begin{equation}\label{eq:dual_gd_comp}
    y_{t+1}^{(ij)} = \arg\min_{y \in \R^d}\bigg\{ \left(\nabla_{y, ij} q_A(x_t, y_t) + \mu_{ij}\nabla f_{ij}^*(\mu_{ij}y_t^{(ij)})\right)^\top y + \frac{p_{ij}}{2\eta} \|y - y_t^{(ij)}\|^2\bigg\},
\end{equation}
where the minimization problem actually has a closed form solution. Yet, as mentioned before, solving Equation~\eqref{eq:dual_gd_comp} requires computing the derivative of $f_{ij}^*$. In order to avoid this, a trick introduced by~\citet{lan2017optimal} and later used in~\citet{wang2017exploiting} is to replace the Euclidean distance term by a well-chosen Bregman divergence. More specifically, the Bregman divergence of a convex function $\phi$ is defined as:
\begin{equation}
    D_\phi(x, y) = \phi(x) - \phi(y) - \nabla \phi(y)^\top (x - y).
\end{equation}
Bregman gradient algorithms typically enjoy the same kind of guarantees as standard gradient algorithms, but with slightly different notions of \emph{relative} smoothness and strong convexity~\citep{bauschke2017descent, lu2018relatively}. Note that the Bregman divergence of the squared Euclidean norm is the squared Euclidean distance, and the standard gradient descent algorithm is recovered in that case. We now replace the Euclidean distance by the Bregman divergence induced by function $\phi: y \mapsto (L_{ij} / \mu_{ij}^2) f_{ij}^*(\mu_{ij} y^{(ij)})$, which is normalized to be $1$-strongly convex since $f_{ij}^*$ is $L_{ij}^{-1}$-strongly convex. We introduce the constant $\alpha > 0$ such that $\mu_{ij}^2 = \alpha L_{ij}$ for all computation edges $(i,j)$. Using the definition of the Bregman divergence with respect to $\phi$, we write:
\begin{align*}
y_{t+1}^{(ij)} &= \arg\min_{y \in \R^d} \left(\nabla_{y, ij} q_A(x_t, y_t) + \mu_{ij} \nabla f_{ij}^*(\mu_{ij} y_t^{(ij)})\right)^\top y + \frac{p_{ij}}{\eta} D_\phi\left(y, y_t^{(ij)}\right)\\
&= \arg \min_{y\in \R} \left(\frac{\alpha \eta}{p_{ij}} \nabla_{y, ij} q_A(x_t, y_t) - \left(1 - \frac{\alpha \eta}{p_{ij}}\right) \mu_{ij} \nabla f_{ij}^*(\mu_{ij}y_t^{(ij)}) \right)^\top y + f_{ij}^*(\mu_{ij}y)\\
&= \frac{1}{\mu_{ij}} \nabla f_{ij}\left(\left(1 - \frac{\alpha \eta}{p_{ij}}\right)\nabla f_{ij}^*(\mu_{ij} y_t^{(ij)}) - \frac{\alpha \eta}{\mu_{ij} p_{ij}} \nabla_{y, ij} q_A(x_t, y_t)\right).
\end{align*}
In particular, if we know $\nabla f_{ij}^*(\mu_{ij}y_t^{(ij)})$ then it is possible to compute $y_{t+1}^{(ij)}$. Besides, 
\begin{equation}
    \nabla f_{ij}^*(\mu_{ij}y_{t+1}^{(ij)}) = (1 - \alpha \eta)\nabla f_{ij}^*(\mu_{ij} y_t^{(ij)}) - \frac{\alpha \eta}{\mu_{ij}} \nabla_{y, ij} q_A(x_t, y_t),
\end{equation}
so we can also compute $\nabla f_{ij}^*(\mu_{ij}y_{t+1}^{(ij)})$, and we can use it for the next step. Therefore, instead of computing a dual gradient at each step, we can simply choose $y_0^{(i)} = \mu_{ij}^{-1} \nabla f_{ij}(z_0^{(ij)})$ for any $z_0^{(ij)}$, and iterate from this. Therefore, the Bregman coordinate update applied to Problem~\eqref{eq:dual_problem} in the block of direction $(i,j)$ with $y_0^{(ij)} = \mu_{ij}^{-1}\nabla f_i(z_0^{(ij)})$ yields:
\begin{equation}
\label{eq:dual-free}
z_{t+1}^{(ij)} = \left(1 - \frac{\alpha \eta}{p_{ij}}\right)z_t^{(ij)} - \frac{\alpha \eta}{p_{ij}\mu_{ij}}\nabla_{y, ij} q_A(x_t, y_t), \qquad y_{t+1}^{(ij)} = \mu_{ij}^{-1} \nabla f_i(z_{t+1}^{(ij)}).
\end{equation}
The iterations of~\eqref{eq:dual-free} are called a \emph{dual-free} algorithm because they are a transformation of the iterations from~\eqref{eq:dual_gd_comp} that do not require computing $\nabla f_{ij}^*$ anymore. This is obtained by replacing the Euclidean distance in~\eqref{eq:dual_gd_comp} by the Bregman divergence of a function proportional to $f_{ij}^*$.

\subsection{Distributed implementation}
Iterations from~\eqref{eq:dual-free} do not involve functions $f_{ij}^*$ anymore, which was our first goal. Yet, they consist in updating dual variables associated with edges of the augmented graph, and have no clear distributed meaning yet. In this section, we rewrite the updates of~\eqref{eq:dual-free} in order to have an easy to implement distributed algorithm. The key steps are (i) multiplication of the updates by $A$, (ii) expliciting the gossip matrix and (iii) remarking that $\theta_t^{(i)} = (\Sigma A(x_t, y_t))^{(i)}$ converges to the primal solution for all $i$. For a vector $z \in \R^{(n + nm)d}$, we denote $[z]_\comm \in \R^{nd}$ its restriction to the communication nodes, and $[M]_\comm \in \R^{nd \times nd}$ similarly refers to the restriction on communication edges of a matrix $M \in \R^{(n + nm)d \times (n + nm)d}$. By abuse of notations, we call $A_\comm \in \R^{nd \times Ed}$ the restriction of $A$ to communication nodes and edges. We denote $P_\comm$ the projector on communication edges, and $P_\comp$ the projector on $y$. We multiply the $x$ (communication) update in~\eqref{eq:dual-free} by $A$ on the left (which is standard~\citep{scaman2017optimal, hendrikx2019accelerated}) and obtain:
\begin{equation}\label{eq:comm_update_almost}
    A_\comm x_{t+1} = A_\comm x_t - \eta p_\comm^{-1} [A P_\comm A^\top]_\comm [\Sigma  A (x_t, y_t)]_\comm.
\end{equation}
Note that $[P_\comm A^\top \Sigma  A (x_t, y_t)]_\comm = [P_\comm A^\top]_\comm [\Sigma  A (x_t, y_t)]_\comm$ because $P_\comm$ and $\Sigma $ are non-zero only for communication edges and nodes. Similarly, and as previously stated, one can verify that $A_\comm[P_\comm A^\top]_\comm = [AP_\comm A^\top]_\comm = W \otimes I_d \in \R^{nd \times nd}$ where $W$ is a gossip matrix. We finally introduce $\tilde{x}_t \in \R^{nd}$ which is a variable associated with nodes, and which is such that $\tilde{x}_t = A_\comm x_t$. With this rewriting, the communication update becomes:
\begin{equation*}
    \tilde{x}_{t+1} = \tilde{x}_t - \eta p_\comm^{-1} (W \otimes I_d) \Sigma_\comm \left[A(x_t, y_t)\right]_\comm.
\end{equation*}
To show that $\left[A(x_t, y_t)\right]_\comm$ is locally accessible to each node, we write:
\begin{align*}
    [A (x_t, y_t)]_\comm^{(i)} &= (A_\comm x_t)^{(i)} - \bigg(\sum_{k=1}^n\sum_{j=1}^m (A (e_{kj} \otimes y_t^{(kj)}))^{(i)}\bigg) = (\tilde{x}_t)^{(i)} - \sum_{j=1}^m \mu_{ij} y_t^{(ij)}.
\end{align*}
We note this rescaled local vector $\theta_t = \Sigma _\comm ([A (x_t, y_t)]_\comm)$, and obtain for variables $\tilde{x}_t$ the gossip update of~\eqref{eq:comm_update}.
Note that we directly write $y_t^{(ij)}$ instead of $P_{ij}y_t^{(ij)}$ even though there has been a multiplication by the matrix $A$. This is allowed because Equation~\eqref{eq:grad_update} implies that (i) $y_t^{(ij)} \in \Ker(f_{ij})^\perp$ for all $t$, and (ii) the value of $(I_d - P_{ij})z_t^{(ij)}$ does not matter since $z_t^{(ij)}$ is only used to compute $\nabla f_{ij}$. We now consider computation edges, and remark that:
\begin{equation}\label{eq:grad_qA_comp}
    \nabla_{y, ij} q_A(x_t, y_t) = - \mu_{ij} (\Sigma _\comm)_{ii} ([A (x_t, y_t)]_\comm)^{(i)} = - \mu_{ij} \theta_t.
\end{equation}
Plugging Equation~\eqref{eq:grad_qA_comp} into the updates of~\eqref{eq:dual-free}, we obtain the following updates:
\begin{equation} \label{eq:comm_update}
    \tilde{x}_{t+1} = \tilde{x}_t - \frac{\eta}{p_\comm} (W \otimes I_d) \theta_t,
\end{equation}
for communication edges, and for the local update of the $j$-th component of node $i$:
\begin{equation}
\label{eq:grad_update}
z_{t+1}^{(ij)} = \left(1 - \frac{\alpha \eta}{p_{ij}}\right)z_t^{(ij)} + \frac{\alpha \eta}{p_{ij}} \theta_t^{(i)}, \qquad \theta_{t+1}^{(i)} = \frac{1}{\sigma_i}\bigg( \tilde{x}_{t+1}^{(i)} - \sum_{j=1}^m \nabla f_{ij}(z_{t+1}^{(ij)})\bigg).
\end{equation}
Finally, Algorithm~\ref{algo:DVR} is obtained by expressing everything in terms of $\theta_t$ and removing variable $\tilde{x}_t$. To simplify notations, we further consider $\theta$ as a matrix in $\R^{n \times d}$ (instead of a vector in $\R^{nd}$), and so the communication update of Equation~\eqref{eq:comm_update} is a standard gossip update with matrix $W$, which we recall is such that $W \otimes I_d = [AP_\comm A^\top]_\comm$.  
We now discuss the local updates of Equation~\eqref{eq:grad_update} more in details, which are closely related to dual-free SDCA updates~\citep{shalev2016sdca}.

\begin{algorithm}
\caption{DVR$(z_0)$}
\label{algo:DVR}
\begin{algorithmic}[1]
\STATE $\alpha = 2\lambda_{\min}^+(A_\comm^\top D_M^{-1} A_\comm)$, $\eta = \min\left( \frac{p_\comm}{\lambda_{\max}(A_\comm^\top \Sigma_\comm A_\comm)}, \frac{p_{ij}}{\alpha(1 + \sigma_i^{-1}L_{ij})} \right)$ \COMMENT{Init.}
\STATE $\theta_0^{(i)} = - (\sum_{j=1}^m \nabla f_{ij}(z_0^{(ij)})) / \sigma_i$. \COMMENT{$z_0$ is arbitrary but not $\theta_0$.}
\FOR[Run for $K$ iterations]{$t=0$ to $K-1$}
\STATE Sample $u_t$ uniformly in $[0, 1]$. \COMMENT{Randomly decide the kind of update}
\IF{$u_t \leq p_\comm$}
\STATE $\theta_{t+1} = \theta_t - \frac{\eta}{p_\comm} \Sigma  W \theta_t$ \COMMENT{Communication using $W$}
\ELSE 
\FOR{$i = 1$ to $n$}
\STATE Sample $j \in \{1, \cdots, m\}$ with probability $p_{ij}$.
\STATE $z_{t+1}^{(ij^\prime)} = z_{t}^{(ij^\prime)}$ for $j \neq j^\prime$ \COMMENT{Only one virtual node is updated}
\STATE $z_{t+1}^{(ij)} = \left(1 - \frac{\alpha \eta}{p_{ij}}\right)z_t^{(ij)} + \frac{\alpha \eta}{p_{ij}} \theta_t^{(i)}$ \COMMENT{Virtual node update}
\STATE $\theta_{t+1}^{(i)} = \theta_t^{(i)} - \frac{1}{\sigma_i}\left(\nabla f_{ij}(z_{t+1}^{(ij)}) - \nabla f_{ij}(z_{t}^{(ij)})\right)$ \COMMENT{Local update using $f_{ij}$}
\ENDFOR
\ENDIF
\ENDFOR
\STATE \textbf{return} $\theta_K$
\end{algorithmic}
\end{algorithm}

\section{Convergence Rate} \label{sec:algo_results}
The goal of this section is to set parameters $\eta$ and $\alpha$ in order to get the best convergence guarantees. We introduce $\kappa_\comm = \gamma \lambda_{\max}(A_\comm^\top \Sigma_\comm A_\comm) / \lambda_{\min}^+(A_\comm^\top D_M^{-1} A_\comm)$, where $\lambda_{\min}^+$ and $\lambda_{\max}$ respectively refer to the smallest non-zero and the highest eigenvalue of the corresponding matrices. We denote $D_M$ the diagonal matrix such that $(D_M)_{ii} = \sigma_i + \lambda_{\max}(\sum_{j=1}^m L_{ij} P_{ij})$, where $\nabla^2 f_{ij}(x) \preccurlyeq L_{ij} P_{ij}$ for all $x \in \R^d$. Note that we use notation $\kappa_\comm$ since it corresponds to a condition number. In particular, $\kappa_\comm \leq \kappa_s$ when $\sigma_i = \sigma_j$ for all $i,j$, and $\kappa_\comm$ more finely captures the interplay between regularity of local functions (through $D_M$ and $\Sigma_\comm$) and the topology of the network (through $A$) otherwise. 

\begin{theorem}\label{thm:dvr_rate}
We choose $p_\comm = \big(1 + \gamma \frac{m + \kappa_s}{\kappa_\comm}\big)^{-1}$, $p_{ij} \propto (1 - p_\comm)(1 + L_{ij} / \sigma_i)$ and $\alpha$ and $\eta$ as in Algorithm~\ref{algo:DVR}. Then, there exists $C_0 > 0$ that only depends on $\theta_0$ (initial conditions) such that for all $t > 0$, the error and the expected time $T_\varepsilon$ required to reach precision $\varepsilon$ are such that:
\begin{equation}
    \sum_{i=1}^n\frac{1}{2}\|\theta^{(i)}_t - \theta^\star\|^2 \leq C_0\left(1 - \frac{\alpha \eta}{2}\right)^t, \hbox{ and so } T_\varepsilon = O\left(\left[2(m + \kappa_s) + \tau \frac{\kappa_\comm}{\gamma}\right]\log\varepsilon^{-1}\right).
\end{equation}
\end{theorem}
\begin{proof}[Proof sketch]
We have seen in Section~\ref{sec:algo_design} that DVR is obtained by applying Bregman coordinate descent on a well-chosen dual problem. Therefore, one of our key results consists in proving convergence rates for Bregman coordinate descent. In order to ease the reading of the paper, we present these results for a general setting in Appendix~\ref{app:bregman_coordinate_gradient}, which is self-contained and which we believe to be of independent interest (beyond its application to decentralized optimization). 

Then, Appendix~\ref{app:application_augmented_problem} focuses on the application to decentralized optimization. In particular, we recall the Equivalence between DVR and Bregman coordinate descent applied to the dual problem of Equation~\eqref{eq:dual_problem}, and show that its structure is suited to the application of coordinate descent. Indeed, no two virtual edges adjacent to the same node are updated at the same time with our sampling. Then, we evaluate the relative smoothness and strong convexity constants of the augmented problem, which allows to derive adequate values for parameters $\alpha$ and $\eta$. Finally, we choose $p_\comm$ in order to minimize the execution time of DVR.  
\end{proof} 

We would like to highlight the fact that the convergence theory of DVR decomposes nicely into several building blocks, and thus simple rates are obtained. This is not so usual for decentralized algorithms, for instance many follow-up papers were needed to obtain a tight convergence theory for EXTRA~\citep{shi2015extra, jakovetic2018unification, xu2020distributed, li2020revisiting}. We now discuss the convergence rate of DVR more in details. 

\textbf{Computation complexity.} The computation complexity of DVR is the same computation complexity as locally running a stochastic algorithm with variance reduction at each node. This is not surprising since, as we argue later, DVR can be understood as a decentralized version of an algorithm that is closely related to dual-free SDCA~\citep{shalev2016sdca}. Therefore, this improves the computation complexity of EXTRA from $O(m( \kappa_b + \gamma^{-1}))$ individual gradients to $O(m + \kappa_s)$, which is the expected improvement for stochastic variance-reduced algorithm. In comparison, GT-SAGA~\citep{xin2020decentralized}, a recent decentralized stochastic algorithm, has a computation complexity of order $O(m + \kappa_s^2 / \gamma^2)$, which is significantly worse than that of DVR, and generally worse than that of EXTRA as well. 

\textbf{Communication complexity.} The communication complexity of DVR (\emph{i.e.}, the number of communications, so the communication time is retrieved by multiplying by $\tau$) is of order $O(\kappa_\comm / \gamma)$, and can be improved to $O(\kappa_\comm / \sqrt{\gamma})$ using Chebyshev acceleration (see Section~\ref{sec:acceleration}). Yet, this is in general worse than the $O(\kappa_b + \gamma^{-1})$ communication complexity of EXTRA or NIDS, which can be interpreted as a partly accelerated communication complexity since the optimal dependence is $O(\sqrt{\kappa_b / \gamma})$~\citep{scaman2019optimal}, and $2\sqrt{\kappa_b / \gamma} = \kappa_b + \gamma
^{-1}$ in the worst case ($\kappa_b = \gamma^{-1}$). Yet, stochastic updates are mainly intended to deal with cases in which the computation time dominates, and we show in the experimental section that DVR outperforms EXTRA and NIDS for a wide range of communication times $\tau$ (the computation complexity dominates roughly as long as $\tau < \sqrt{\gamma}(m + \kappa_s) / \kappa_\comm)$. Finally, the communication complexity of DVR is significantly lower than that of DSA and GT-SAGA, the primal decentralized stochastic alternatives presented in Section~\ref{sec:introduction}.

\textbf{Homogeneous parameter choice.} In the homogeneous case ($\sigma_i = \sigma_j$ for all $i,j$), choosing the optimal $p_\comp$ and $p_\comm$ described above leads to $\eta \lambda_{\max}(W) = \sigma p_\comm$. Therefore, the communication update becomes $\theta_{t+1} = \left(I - W / \lambda_{\max}(W)\right)\theta_t$, which is a gossip update with a standard step-size (independent of the optimization parameters). Similarly, $\alpha \eta(m + \kappa_s) = p_\comp$, and so the step-size for the computation updates is independent of the network.

\textbf{Links with SDCA.}
The single-machine version of Algorithm~\ref{algo:DVR} ($n=1$, $p_\comm = 0$) is closely related to dual-free SDCA~\citep{shalev2016sdca}. The difference is in the stochastic gradient used: DVR uses $\nabla f_{ij}(z_t^{(ij)})$, where $z_t^{(ij)}$ is a convex combination of $\theta_k^{(i)}$ for $k<t$, whereas dual-free SDCA uses $g_t^{(ij)}$, which is a convex combination of $\nabla f_{ij}(\theta_k^{(i)})$ for $k<t$. Both algorithms obtain the same rates. 

\textbf{Local synchrony.}
Instead of using the synchronous communications of Algorithm~\ref{algo:DVR}, it is possible to update edges one at a time, as in~\citet{hendrikx2019accelerated}.
This can be very efficient in heterogeneous settings (both in terms of computation and communication times) and similar convergence results can be obtained using the same framework, and we leave the details for future work.

\section{Acceleration}
\label{sec:acceleration}
We show in this section how to modify DVR to improve the convergence rate of Theorem~\ref{thm:dvr_rate}.

\textbf{Network acceleration.} 
Algorithm~\ref{algo:DVR} depends on $\gamma^{-1}$, also called the \emph{mixing time} of the graph, which can be as high as $O(n^2)$ for a chain of length $n$~\citep{mohar1997some}. However, it is possible to improve this dependency to $\gamma^{-1/2}$ by using Chebyshev acceleration, as in~\citet{scaman2017optimal}. To do so, the first step is to choose a polynomial $P$ of degree $k$ and communicate with $P(W)$ instead of $W$. In terms of implementation, this comes down to performing $k$ communication rounds instead of one, but this makes the algorithm depend on the spectral gap of $P(W)$. Then, the important fact is that there is a polynomial $P_{\gamma}$ of degree $\lceil \gamma^{-1/2} \rceil$ such that the spectral gap of $P_\gamma(W)$ is of order $1$. Each communication step with $P_\gamma(W)$ only takes time $\tau {\rm deg}(P_\gamma) = \tau \lceil\gamma^{-1/2}\rceil$, and so the communication term in Theorem~\ref{thm:dvr_rate} can be replaced by $\tau \kappa_\comm \gamma^{-1/2}$, thus leading to \emph{network acceleration}. The polynomial $P_\gamma$ can for example be chosen as a Chebyshev polynomial, and we refer the interested reader to~\citet{scaman2017optimal} for more details. Finally, other polynomials yield even faster convergence when the graph topology is known~\citep{berthier2020accelerated}.

\textbf{Catalyst acceleration.}
Catalyst~\citep{lin2015universal} is a generic framework that achieves acceleration by solving a sequence of subproblems. Because of space limitations, we only present the accelerated convergence rate without specifying the algorithm in the main text. Yet, only mild modifications to Algorithm~\ref{algo:DVR} are required to obtain these rates, and the detailed derivations and proofs are presented in Appendix~\ref{app:catalyst}.

\begin{theorem} \label{thm:rate_dvr_catalyst}
DVR can be accelerated using catalyst, so that the time $T_\varepsilon$ required to reach precision $\varepsilon$ is equal (up to log factors) to
$T_\varepsilon = \tilde{O}\big(\big[m + \sqrt{m\kappa_s} + \tau \sqrt{\kappa_\comm / \gamma} \sqrt{m \kappa_\comm / \kappa_s}\big]\log\varepsilon^{-1}\big)
$.
\end{theorem}

This rate recovers the computation complexity of optimal finite sum algorithms such as 
ADFS~\citep{hendrikx2019accelerated, hendrikx2020optimal}. Although the communication time is slightly increased (by a factor $\sqrt{m \kappa_\comm / \kappa_s}$), ADFS uses a stronger oracle than DVR (proximal operator instead of gradient), which is why we develop DVR in the first place. Although both ADFS and DVR are derived using the same dual formulation, both the approach and the resulting algorithms are rather different: ADFS uses accelerated coordinate descent, and thus has strong convergence guarantees at the cost of requiring dual oracles. DVR uses coordinate descent with the Bregman divergence of $\phi_{ij} \propto f_{ij}^*$ in order to work with primal oracles, but thus loses direct acceleration, which is recovered through the Catalyst framework. Note that the parameters of accelerated DVR can also be set such that $T_\varepsilon = \tilde{O}\left(\sqrt{\kappa_\comm}\left[m + \tau / \sqrt{\gamma}\right]\log\varepsilon^{-1}\right)$, which recovers the convergence rate of optimal batch algorithms, but loses the finite-sum speedup. 

\section{Experiments} 
\label{sec:experiments}
We investigate in this section the practical performances of DVR. We solve a regularized logistic regression problem on the RCV1 dataset~\citep{lewis2004rcv1} ($d=47236$) with $n=81$ (leading to $m=2430$) and two different graph topologies: an Erd\H{o}s-R\'enyi random graph (see, \emph{e.g.},~\citep{bollobas2001random}) and a grid. We choose $\mu_{k\ell}^2 = 1/2$ for all communication edges, so the gossip matrix $W$ is the Laplacian of the graph. 

\begin{figure}
\subfigure[Erd\H{o}s-R\'enyi, $\sigma=m \cdot 10^{-5}$
]{
    \includegraphics[width=0.32\linewidth]{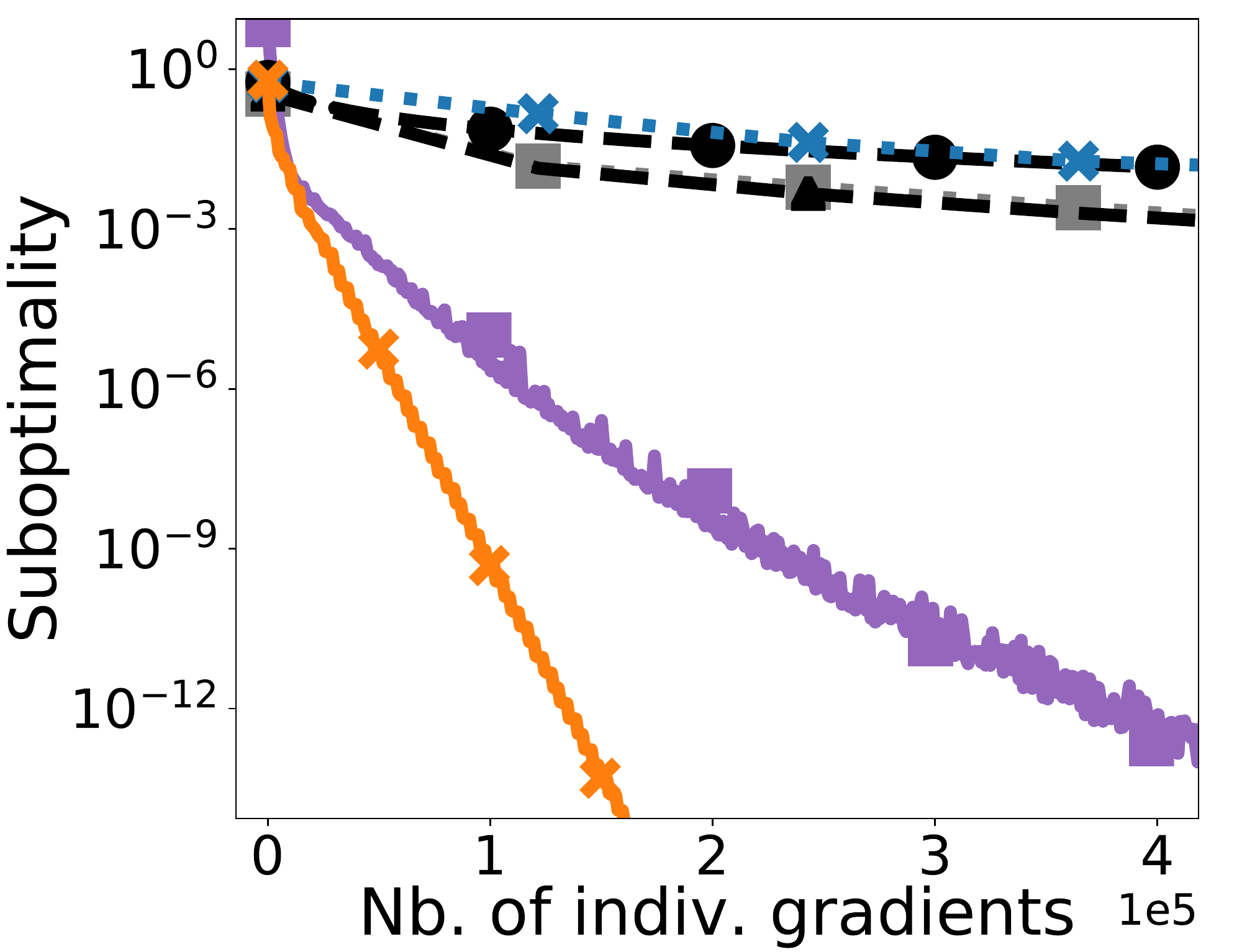} \includegraphics[width=0.32\linewidth]{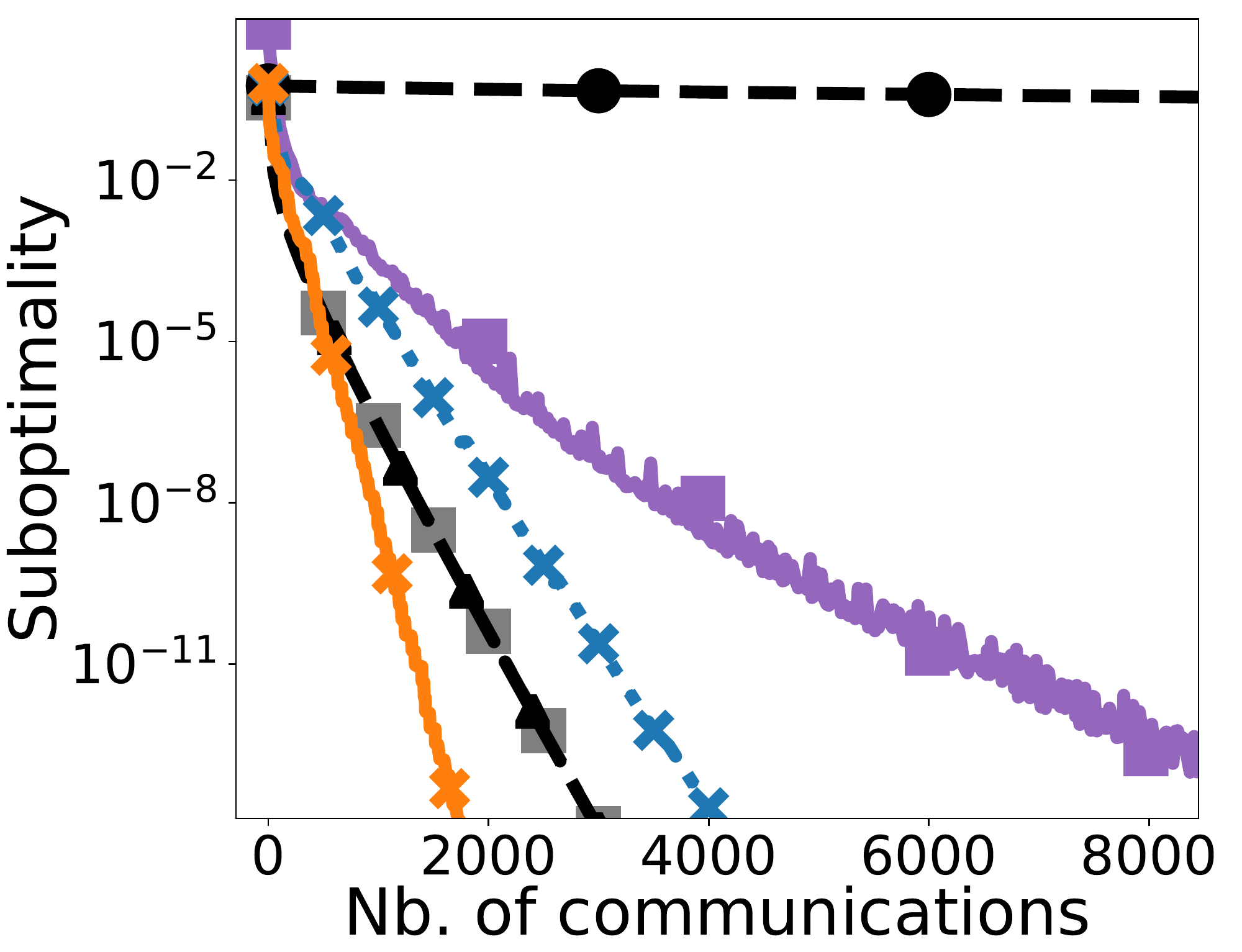}
    \includegraphics[width=0.32\linewidth]{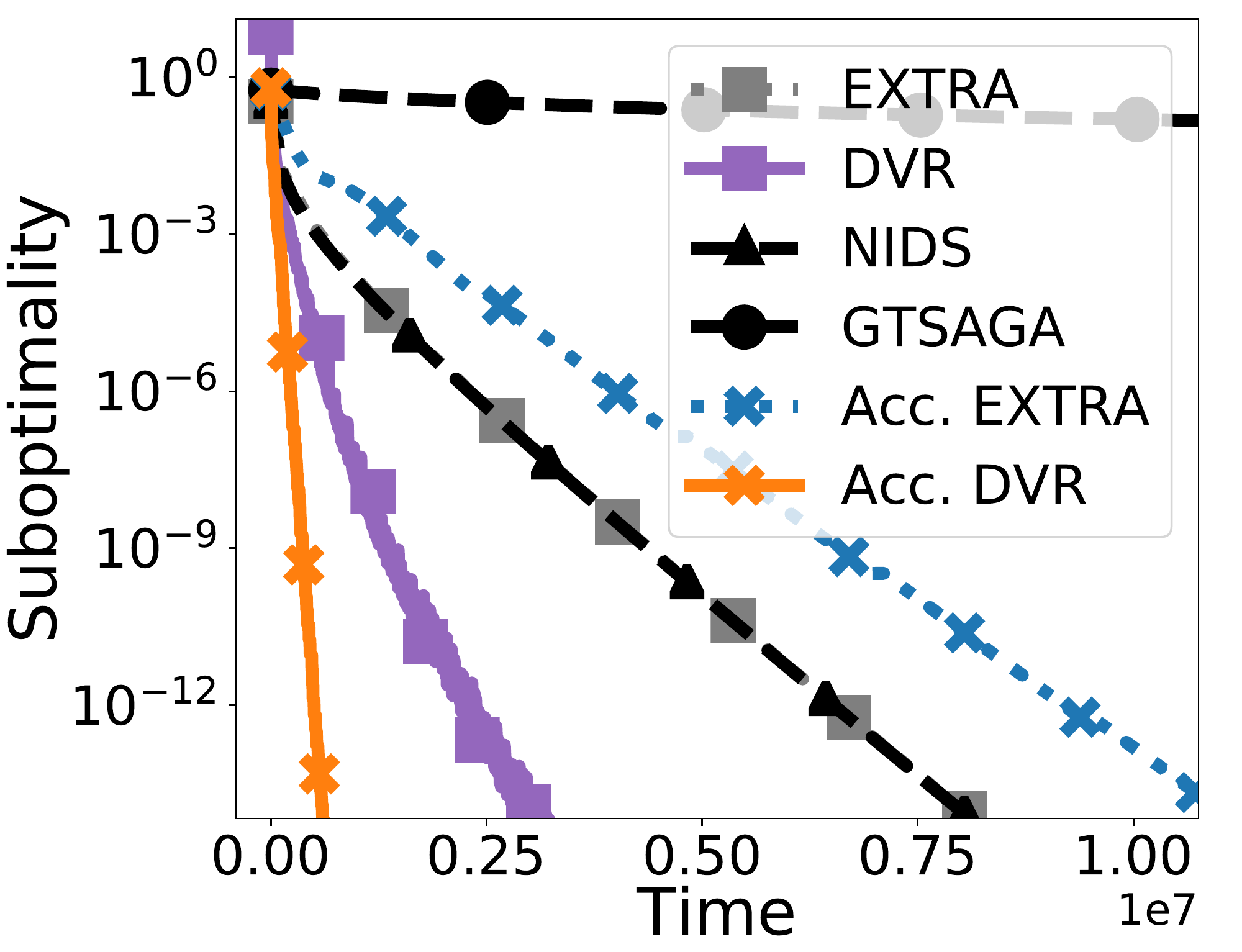}
    \label{fig:erdos_renyi}
}\\
\subfigure[Grid, $\sigma=m \cdot 10^{-5}$
]{
    \includegraphics[width=0.31\linewidth]{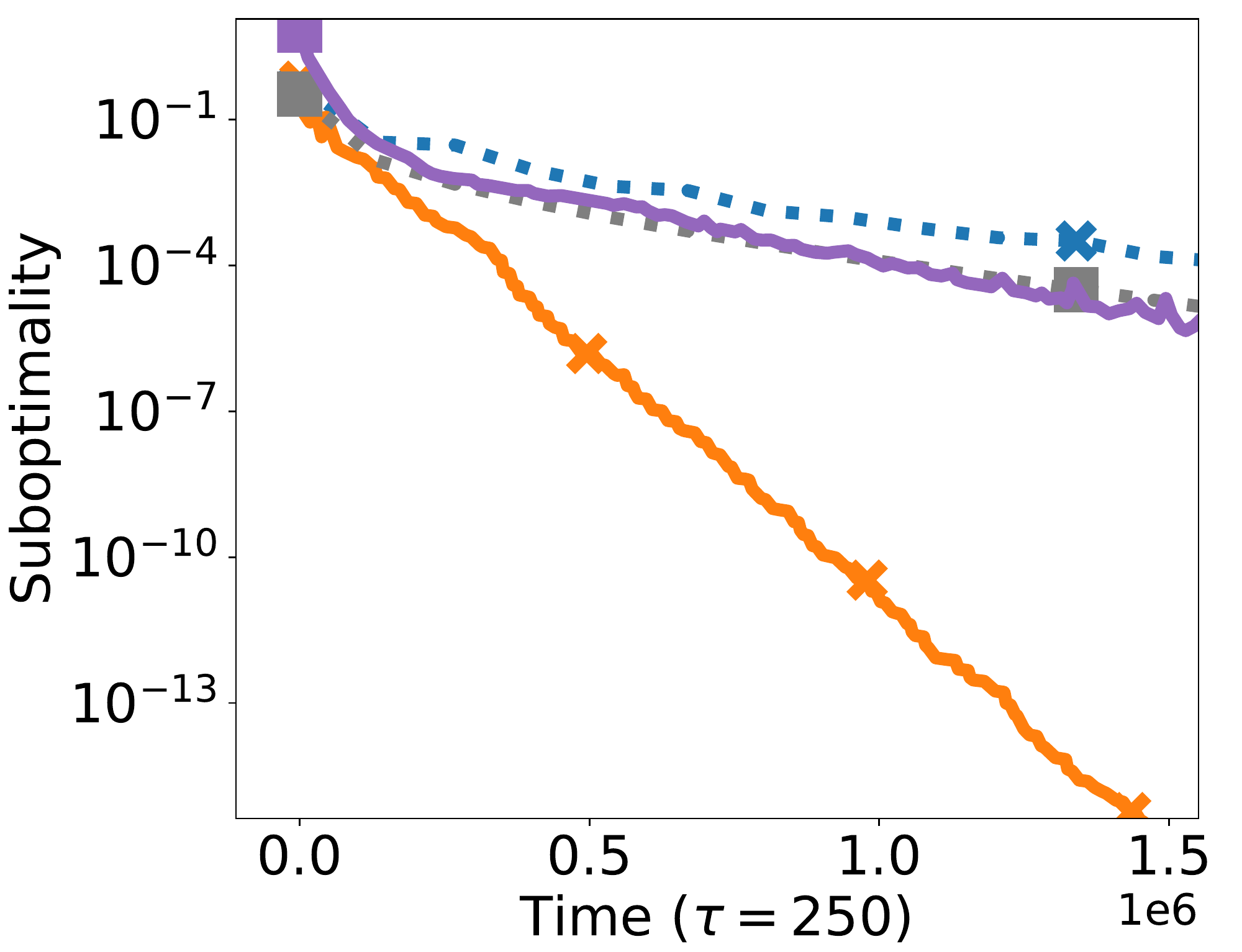}
    \label{fig:grid_c5}
}
\subfigure[Erd\H{o}s-R\'enyi, $\sigma=m \cdot 10^{-7}$]{
    \includegraphics[width=0.31\linewidth]{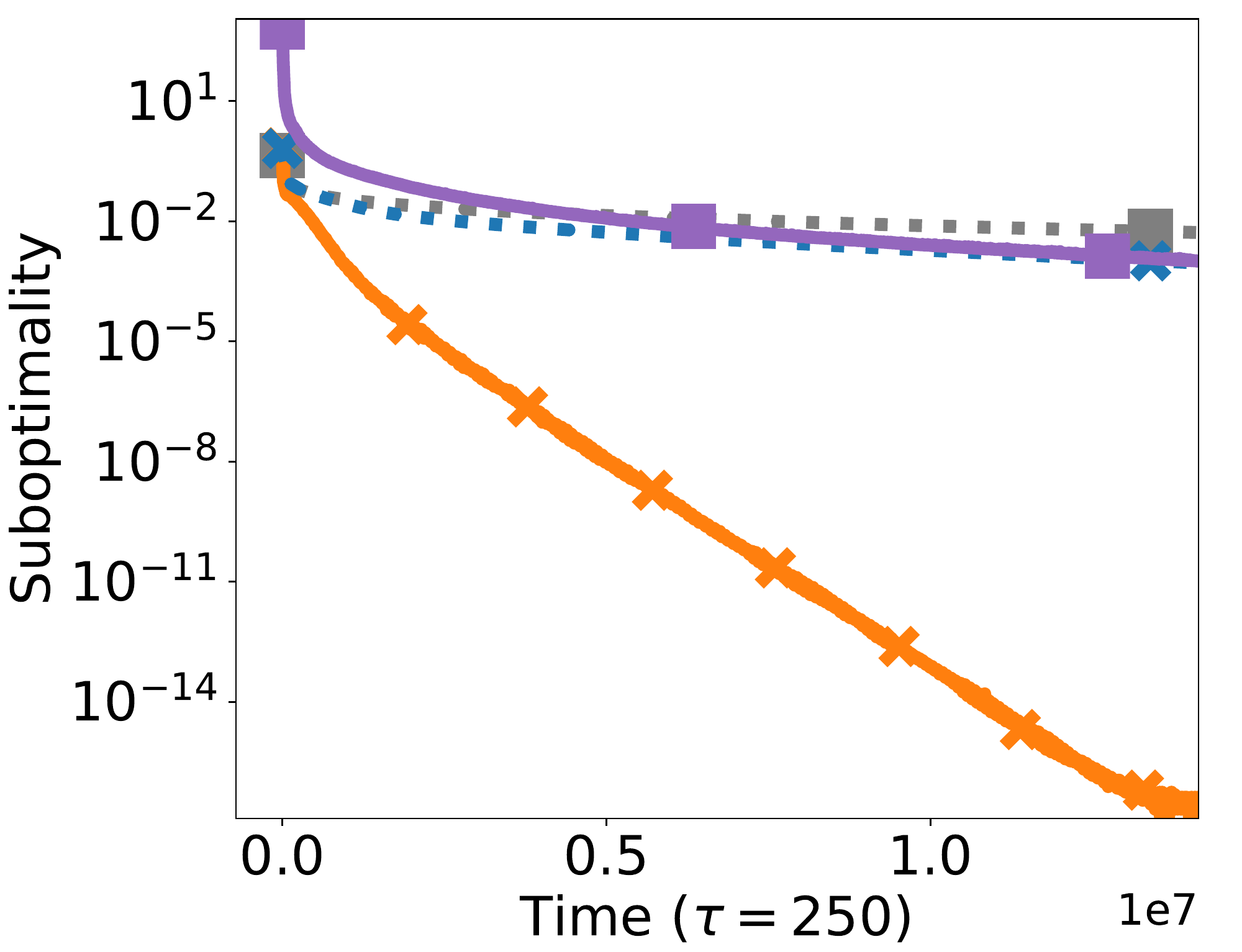}
    \label{fig:grid_c7}
}
\subfigure[Grid, $\sigma=m \cdot 10^{-7}$]{
    \includegraphics[width=0.31\linewidth]{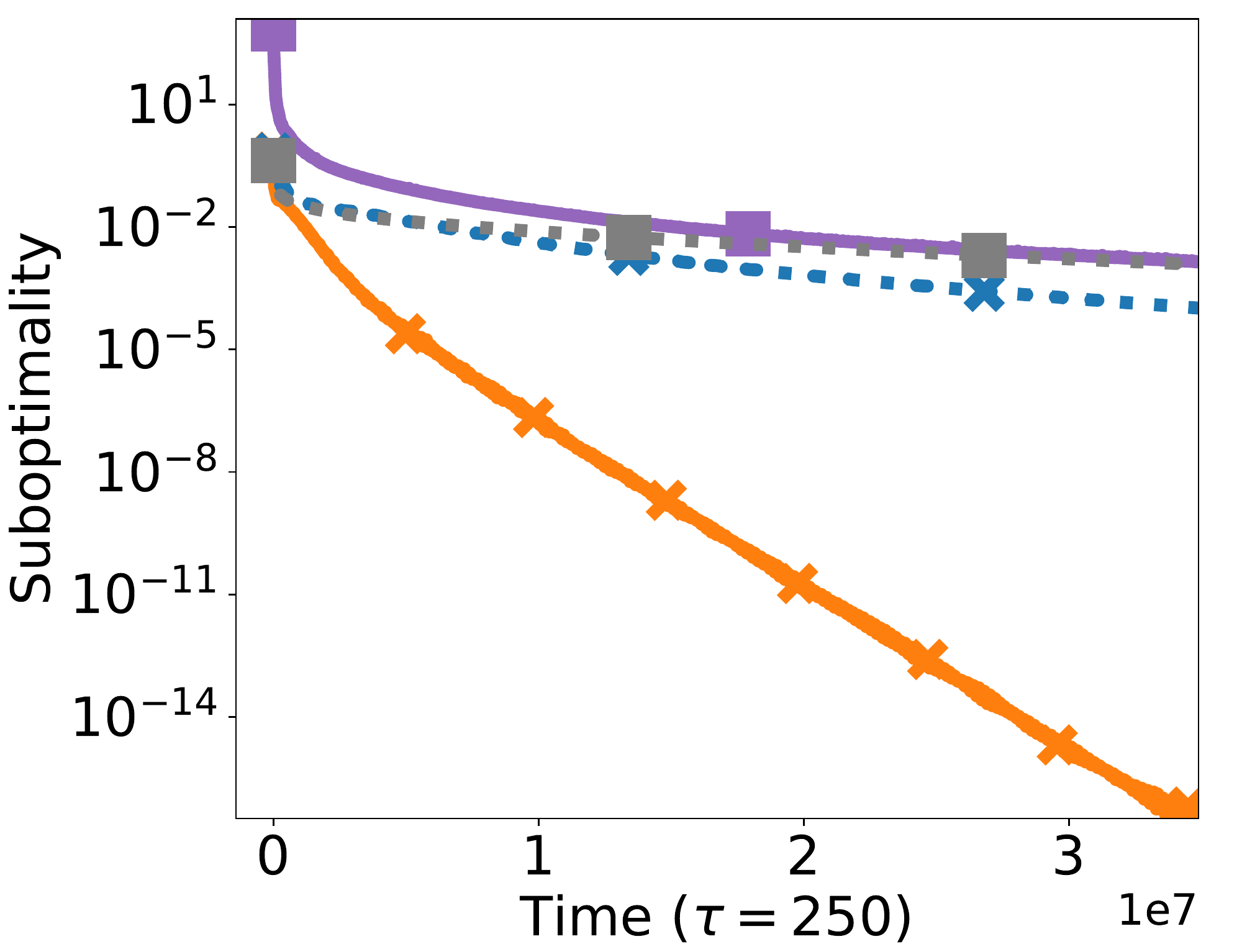}
    \label{fig:er_c7}
}
\caption{Experimental results for the RCV1 dataset with different graphs of size $n=81$, with $m=2430$ samples per node, and with different regularization parameters. \label{fig:plots}}
\end{figure}

Figure~\ref{fig:plots} compares the performance of DVR with that of state-of-the-art primal algorithms such as EXTRA~\citep{shi2015extra}, NIDS~\citep{li2019decentralized}, GT-SAGA~\citep{xin2020decentralized}, and Catalyst accelerated versions of EXTRA~\citep{li2020revisiting} and DVR. Suboptimality refers to $F(\theta_t^{(0)}) - F(\theta^\star)$, where node $0$ is chosen arbitrarily and $F(\theta^\star)$ is approximated by the minimal error over all iterations. Each subplot of Figure~\ref{fig:erdos_renyi} shows the same run with different x axes. The left plot measures the complexity in terms of individual gradients ($\nabla f_{ij}$) computed by each node whereas the center plot measures it in terms of communications (multiplications by $W$). All other plots are taken with respect to (simulated) time (\emph{i.e.}, computing $\nabla f_{ij}$ takes time $1$ and multiplying by $W$ takes time $\tau$) with $\tau = 250$ in order to report results that are independent of the computing cluster hardware and status. All parameters are chosen according to theory, except for the smoothness of the $f_i$, which requires finding the smallest eigenvalue of a $d \times d$ matrix. For this, we start with $L_b = \sigma_i + \sum_{j=1}^m L_{ij}$ (which is a known upper bound), and decrease it while convergence is ensured, leading to $\kappa_b = 0.01 \kappa_s$. The parameters for accelerated EXTRA are chosen as in~\citet{li2020revisiting} since tuning the number of inner iterations does not significantly improve the results (at the cost of a high tuning effort). For accelerated DVR, we set the number of inner iterations to $N / p_\comp$ (one pass over the local dataset).
We use Chebyshev acceleration for (accelerated) DVR but not for (accelerated) EXTRA since it is actually slower, as predicted by the theory.

As expected from their theoretical iteration complexities, NIDS and EXTRA perform very similarly~\citet{li2020revisiting}, and GT-SAGA is the slowest method. Therefore, we only plot NIDS and GT-SAGA in Figure~\ref{fig:erdos_renyi}. We then see that though it requires more communications, DVR has a much lower computation complexity than EXTRA, which illustrates the benefits of stochastic methods. We see that DVR is faster overall if we choose $\tau=250$, and both methods perform similarly for $\tau \approx 1000$, at which point communicating takes roughly as much time as computing a full local gradient. We then see that accelerated EXTRA has quite a lot of overhead and, despite our tuning efforts, is slower than EXTRA when the regularization is rather high. On the other hand, accelerated DVR consistently outperforms DVR by a relatively large margin. The communication complexity is in particular greatly improved, allowing accelerated DVR to be the fastest method regardless of the setting.
Further experimental results are given in Appendix~\ref{app:experiments}, and the code is available in supplementary material.

\section{Conclusion}
This paper introduces DVR, a Decentralized stochastic algorithm with Variance Reduction obtained using Bregman block coordinate descent on a well-chosen dual formulation. Thanks to this approach, DVR inherits from the fast rates and simple theory of dual approaches without the computational burden of relying on dual oracles. Therefore, DVR has a drastically lower computational cost than standard primal decentralized algorithms, although sometimes at the cost of a slight increase in communication complexity. The framework used to derive DVR is rather general and could in particular be extended to analyze asynchronous algorithms. Finally, although deriving a direct acceleration of DVR is a challenging open problem, Catalyst and Chebyshev accelerations allow to significantly reduce DVR's communication overhead both in theory and in practice. 

\section*{Acknowledgements}

This work was funded in part by the French government under management of Agence Nationale de la Recherche as part of the ``Investissements d'avenir'' program, reference ANR-19-P3IA-0001 (PRAIRIE 3IA Institute). We also acknowledge support from the European Research Council (grant SEQUOIA 724063) and from the MSR-INRIA joint centre. 

\bibliographystyle{plainnat}
\bibliography{biblio}

\newpage

\appendix

This appendix contains the details of the derivations and proofs from the main text. More specifically, Appendix~\ref{app:bregman_coordinate_gradient} is a self-contained appendix that specifies the Bregman coordinate descent algorithm and proves its convergence rate. Appendix~\ref{app:application_augmented_problem} focuses on the application of Bregman coordinate descent to the dual problem (relative smoothness and strong convexity constants, sparsity structure), and how to retrieve guarantees on the primal parameters. Appendix~\ref{app:catalyst} is devoted to presenting the Catalyst acceleration of DVR and proving its convergence speed, and Appendix~\ref{app:experiments} details the experimental setting, along with more experiments. 

\section{Block Coordinate descent}
\label{app:bregman_coordinate_gradient}
We focus in this section on the general problem minimizing $f + g$ using coordinate Bregman gradient, where $g$ is separable, \emph{i.e.}, $g(x) = \sum_{i=1}^d g_i(x^{(i)})$. This is a self-contained section, and notations may differ from the rest of the paper. In particular, function $f$ is for now arbitrary and not related to $F$ or $f_i$ from Problem~\eqref{eq:general_problem}, and the dimension $d$ is arbitrary as well.

We first precise the blocks sampling rule. More specifically, we define a block $b \subset \{1, \dots, d\}$ as a collection of coordinates, and $\cB$ is the set of all blocks that can be chosen for the updates. Then, the algorithm updates each block $b \in \cB$ with probability $p(b)$, so that the probability of updating a given coordinate is given by $p_i = \sum_{i \in b} p(b)$. Similarly to individual coordinates, we write $x^{(b)}$ the restriction of $x$ to coordinates in $b$. The Bregman coordinate gradient update for a block of coordinates $b$ writes:
\begin{equation}\label{eq:coord_breg_constrained}
    x_{t+1} = \arg \min_{x \in \R^d} \left\{V_t^b(x) \triangleq \sum_{i \in b} \frac{\eta_t}{p_i} \left[\nabla_i f(x_t)^\top x  + g_i(x^{(i)})\right] + D_{\phi}(x, x_t)\right\},
\end{equation}
where $\nabla_i f$ denotes the gradient of $f$ in direction $i$. Note that this update is more general than the one used to derive DVR, for which $g = 0$.
In order to derive strong guarantees for this block coordinate descent algorithm, we need to ensure that there is some separability in functions $f$ and $\phi$, and that the block structure is suited to this separability. All the assumptions about the separability structure of $f$, $g$ and $\phi$ are contained in the following assumption.
\begin{assumption}[Separability]\label{ass:separability}
The function $g$ is separable and the function $\phi$ is block-separable for $b$, meaning that for all $b \in \cB$, there exist two convex functions $\phi_b$ and $\phi_b^\perp$ such that for all $x$,
\begin{equation} \label{eq:block_separability_phi}
    \phi(x) = \phi_b(x^{(b)}) + \phi_b^\perp(x - x^{(b)}).
\end{equation}
Besides, for all $b \in \cB$, either of the following two hold:
\begin{enumerate}
    \item $\phi$ and $f$ are separable for $b$, \emph{i.e.}, $\phi_b(x^{(b)}) = \sum_{i \in b} \phi_i(x^{(i)})$, and \[\sum_{i \in b} \left[f(x_t + \delta_i e_i) - f(x_t)\right] = f\left(x_t + \sum_{i\in b} \delta_i e_i \right) - f(x_t).\]
    \item $p_i = p_j$ for all $i,j \in b$.
\end{enumerate}
\end{assumption}

If $\phi$ is not block-separable, the support of the Bregman update in direction $b$ may not restricted to $b$. This causes some of the derivations below to fail, which is why we prevent it by assuming that Equation~\eqref{eq:block_separability_phi} holds.

Then, the first option ensures that within a block, the updates do not affect each other. The function $f$ is not separable, but some directions can be updated independently from others. To have these independent updates, we also need to assume further separability of $\phi$ within the blocks. The second option states that if only block-separability of $\phi$ is assumed then within each block for which $\phi$ and $f$ are not separable, coordinates must be picked with the same probability.

Assumption~\ref{ass:separability} is a bit technical but we actually require all statements in order to derive DVR. In particular, the first option is verified when updating within the same block virtual edges that are adjacent to different nodes in the dual problem. The second option is verified when picking all communication edges at once within the same block.

Now that we have made assumptions on the structure of $f$, $g$ and $\phi$, we will make assumptions on their regularity. We start by a directional relative smoothness assumption between $f$ and $\phi$, \emph{i.e.}, we assume that for all $i$, there exists $L_\rel^i$ such that for all $\delta > 0$ and $e_i$ the unit vector of direction $i$,
\begin{equation} \label{eq:relative_smoothness}
    D_f(x + \delta e_i, x) \leq L_\rel^i D_\phi(x + \delta e_i, x).
\end{equation}
Similarly, for $\sigma_\rel > 0$, $f$ is said to be $\sigma_\rel$-strongly convex relatively to $\phi$ if for all $x, y$:
\begin{equation} \label{eq:relative_strong_convexity}
    D_f(x, y) \geq \sigma_\rel D_\phi(x, y).
\end{equation}
We finally assume that $f$ and $\phi$ are convex (but not necessarily smooth). We can now state the central theorem of this section:
\begin{theorem} \label{thm:bregman_coordinate}
Let $f$ and $\phi$ be such that $f$ is $L_\rel^i$-smooth in direction $i$ and $\sigma_\rel$-strongly convex relatively to $\phi$. Denote $p_{\min} = \min_i p_i$, and 
\[L_t = D_\phi(x, x_{t}) + \frac{\eta_t}{p_{\min}}\left(F(x_{t}) - F(x)\right).\]
Then, if the blocks $\cB$ respect Assumption~\ref{ass:separability} (separability) and $\eta_t L_\rel^i < p_i$ for all $i$, the Bregman coordinate descent algorithm guarantees for all $x$:
\begin{align*}
   \esp{L_{t+1}} \leq (1 - \eta_t \sigma_\rel) L_t.
\end{align*}
The same result holds with $L^\prime_t = D_\phi(x, x_{t}) + \frac{1}{L_\rel^{\max}}\left(F(x_{t}) - F(x)\right)$, where $L_\rel^{\max} = \max_i L_\rel^i$.
\end{theorem}

To prove this theorem, we start by proving the monotonicity of such iterations.
\begin{lemma}[Monotonicity]\label{lemma:monotonicity}
We note $\delta_i = e_i^\top(x_{t+1} - x_t) e_i$. If $x_{t+1} = \arg \min_x V_t^b(x)$ then:
\begin{enumerate}
    \item If $\phi$ and $f$ are separable for $b$ then for all $i \in b$, if $\eta_t L_\rel^i \leq p_i$ then $F(x_t) \geq F(x_t + \delta_i)$.
    \item If $p_i = p_j$ for all $i,j \in b$ and $\eta_t L_\rel^b \leq p_b$ then $F(x_t) \geq F(x_{t+1})$.
\end{enumerate}
\end{lemma}

\begin{proof}
We start by the first point. If $\phi$ is separable for $b$ then this means that each coordinate is updated independently. By definition of $x_{t+1}^{(i)}$, we have $V_t^b(x_{t+1}^{(b)}) \leq V_t^b(x_t)$. This writes, splitting over each $i$ and using the fact that $D_{\phi_i}(x_t, x_t) = 0$:
\begin{align*}
     g_i(x_t^{(i)}) - g_i(x_{t+1}^{(i)}) &\geq \nabla_i f(x_t)^\top (x_{t+1}^{(i)} - x_t^{(i)}) + \frac{p_i}{\eta}D_{\phi_i}(x_{t+1}^{(i)}, x_t^{(i)})\\
     &= \nabla_i f(x_t)^\top (x_t + \delta_i - x_t) + \frac{p_i}{\eta}D_{\phi_i}(x_{t+1}^{(i)}, x_t^{(i)})\\
    &= f(x_t + \delta_i) - f(x_t) - D_{f}(x_{t+1}^{(i)}, x_t^{(i)}) + \frac{p_i}{\eta}D_{\phi_i}(x_{t+1}^{(i)}, x_t^{(i)})\\
    &\geq f(x_t + \delta_i) - f(x_t) + \left(\frac{p_i}{\eta} - L_\rel^i \right)D_{\phi_i}(x_{t+1}^{(i)}, x_t^{(i)})\\
    &\geq f(x_t + \delta_i) - f(x_t).
\end{align*}
The result follows from summing over all $i \in b$, and using Assumption~\ref{ass:separability}. For the second point, it is not possible to split the update per coordinate since $\phi$ is not separable. Yet, we can still write (using separability of $g$): 
\begin{equation} \label{eq:monotonicity_2}
    \sum_{i \in b} \frac{\eta_t}{p_i} \left[g(x_t^{(i)}) - g(x_{t+1}^{(i)}) - \nabla_i f(x_t)^\top (x_{t+1}^{(i)} - x_t^{(i)})\right]\geq D_\phi(x_{t+1}, x_t).
\end{equation}
Since $g$ is separable and $p_i = p_b$ for all $i \in b$, Equation~\eqref{eq:monotonicity_2} writes:
\begin{equation} \label{eq:monotonicity_3}
    g(x_t) - g(x_{t+1}) \geq  \nabla f(x_t)^\top (x_{t+1} - x_t) +  \frac{p_b}{\eta}D_\phi(x_{t+1}, x_t).
\end{equation}
Note that this crucially relies on $x_{t+1} - x_t$ having support on $b$, which is enforced by the block-separability of $\phi$. Then, the proof is similar to that of the first point, using that $\eta_t L_\rel^b \leq p_b$.
\end{proof}

Using this monotonicity result allows us to prove Theorem~\ref{thm:bregman_coordinate}.

\begin{proof}[Proof of Theorem~\ref{thm:bregman_coordinate}]
First note that by convexity of all $g_i$,
$$\nabla^2 V_t^b(x) = \sum_{i \in b} \frac{\eta_t}{p_i} \nabla^2 g_i(x^{(i)}) + \nabla^2 \phi(x) \succcurlyeq \nabla^2 \phi(x).$$ 
Therefore, we have $D_{V_t^b}(x,y) \geq D_\phi(x,y)$ for all $x,y \in \R^d$. Applying this with $y = x_{t+1}$ yields:
\begin{equation} \label{eq:sc_Vt}
    V_t^b(x) - V_t^b(x_{t+1}) - \nabla V_t^b(x_{t+1})^\top (x - x_{t+1}) \geq D_\phi(x, x_{t+1}).
\end{equation}
Then, $\nabla V_t^b(x_{t+1}) = 0$ by definition of $x_{t+1}$, so Equation~\eqref{eq:sc_Vt} writes:
\begin{align*}
    D_\phi(x, x_{t+1}) + \sum_{i\in b}\frac{\eta_t}{p_i}\left(g_i(x_{t+1}^{(i)}) - g(x^{(i)})\right) \leq \sum_{i \in b}&\frac{\eta_t}{p_i} \nabla_i f(x_t)^\top (x - x_{t+1})\\
    &+ D_\phi(x, x_t) - D_\phi(x_{t+1}, x_t).
\end{align*}
We first consider that the first option of Assumption~\ref{ass:separability} holds, \emph{i.e.}, that $f$ and $\phi$ are separable in $b$. We note $\delta_i = e_i^\top(x_{t+1} - x_t) e_i$, so that:
\begin{align*}
    - \nabla_i f(x_t)^\top (x_{t+1} - x_t) &= \nabla f(x_t)^\top (x_t + \delta_i - x_t)\\
    &= f(x_t) - f(x_t + \delta_i) + D_f(x_t + \delta_i, x_t)\\
    &\leq f(x_t) - f(x_t + \delta_i) + L_\rel^i D_{\phi_i}(x_{t+1}^{(i)}, x_t^{(i)}).
\end{align*}
Therefore, if $\eta_t L_\rel^i \leq p_i$ for all $i \in b$,
\begin{align*}
    - \sum_{i \in b}& \frac{\eta_t}{p_i}\nabla_i f(x_t)^\top (x_{t+1} - x_t) - D_\phi(x_{t+1}, x_t)\\
    &\leq \sum_{i \in b} \frac{\eta_t}{p_i}\left[f(x_t) - f(x_t + \delta_i)\right] + \sum_{i \in b}\left(\frac{\eta_t L_\rel^i}{p_i} - 1\right)D_{\phi_i}(x_{t+1}^{(i)}, x_t^{(i)})\\
    &\leq \sum_{i \in b} \frac{\eta_t}{p_i}\left[f(x_t) - f(x_t + \delta_i)\right]
\end{align*}
The $g_i(x_{t+1}^{(i)}) - g_i(x^{(i)})$ term can be replaced by $g(x_t + \delta_i) - g(x_t) + g_i(x_t^{(i)}) - g_i(x^{(i)})$ since $g_j(x_{t+1}) = g_j(x_t)$ for $j \neq i$. Therefore, we obtain:
\begin{align}
\begin{split}\label{eq:before_monoton}
    &D_\phi(x, x_{t+1}) + \sum_{i \in b} \frac{\eta_t}{p_i}\left[F(x_t + \delta_i) - F(x_t)\right] + \sum_{i \in b}\frac{\eta_t}{p_i}\left(g_i(x_{t}^{(i)}) - g_i(x^{(i)})\right) \\
    & \leq \sum_{i\in b}\frac{\eta_t}{p_i}\nabla_i f(x_t)^\top (x - x_t) + D_\phi(x, x_t).
    \end{split}
\end{align}
The separability of $F$ in $b$ and its monotonicity lead to, using the fact that $x_{t+1} = x_t + \sum_{i \in b} \delta_i$:
\[\sum_{i \in b} \frac{\eta_t}{p_i}\left[F(x_t + \delta_i) - F(x_t)\right] \geq \frac{\eta_t}{p_{\min}} \sum_{i \in b} \left[F(x_t + \delta_i) - F(x_t)\right] =  \frac{\eta_t}{p_{\min}} \left[F(x_{t+1}) - F(x_t)\right].\]
Therefore, if the first option of Assumption~\ref{ass:separability} holds, we obtain: 
\begin{align}
\begin{split}\label{eq:after_monoton}
    &D_\phi(x, x_{t+1}) +  \frac{\eta_t}{p_{\min}} \left[F(x_{t+1}) - F(x_t)\right] + \sum_{i \in b}\frac{\eta_t}{p_i}\left(g_i(x_{t}^{(i)}) - g_i(x^{(i)})\right) \\
    & \leq \sum_{i\in b}\frac{\eta_t}{p_i}\nabla_i f(x_t)^\top (x - x_t) + D_\phi(x, x_t).
    \end{split}
\end{align}
If the second option holds, \emph{i.e.}, $p_i = p$ for all $i \in b$, then 
\[\sum_{i \in b} \frac{\eta_t}{p_i}\nabla_i f(x_t)^\top (x_{t+1} - x_t) = \frac{\eta_t}{p} \nabla f(x_t)^\top (x_{t+1} - x_t),\]
and Equation~\eqref{eq:after_monoton} can be obtained through similar derivations (at the block-level). Using the separability of $g$, we obtain that \[\esp{\sum_{i \in b}\frac{1}{p_i}\left(g_i(x_{t}^{(i)}) - g_i(x^{(i)})\right)} = g(x_t) - g(x).\]
Then, since $\esp{\sum_{i\in b}\frac{1}{p_i}\nabla_i f(x_t)} = \sum_i p_i^{-1} \sum_{b: i \in b}p(b) \nabla_i f(x_t) = \nabla f(x_t)$, and the relative strong convexity assumption yields:
\[\esp{\sum_{i\in b}\frac{1}{p_i}\nabla_i f(x_t)^\top (x - x_t)} = \nabla f(x_t)^\top (x - x_t) \leq f(x) - f(x_t) - \sigma_\rel D_\phi(x, x_t).\]
Therefore, taking the expectation of Equation~\eqref{eq:before_monoton} yields:
\begin{align*}
    \esp{D_\phi(x, x_{t+1}) + \frac{\eta_t}{p_{\min}}\left(F(x_{t+1}) - F(x_t)\right)} \leq \eta_t\left(F(x) - F(x_t)\right) + (1 -  \eta_t \sigma_\rel) D_\phi(x, x_t).
\end{align*}
We obtain after some rewriting:
\begin{align*}
    &\esp{D_\phi(x, x_{t+1}) + \frac{\eta_t}{p_{\min}}\left(F(x_{t+1}) - F(x)\right)}\\
    &\leq (1 - p_{\min})\frac{\eta_t}{p_{\min}}\left(F(x_{t}) - F(x)\right) + (1 -  \eta_t \sigma_\rel) D_\phi(x, x_t).
\end{align*}
Finally, $\sigma_\rel \leq L_\rel^i$ so $\eta_t\sigma_\rel \leq \eta_t L_\rel^i \leq p_i$ for all $i$, and in particular $1 - p_{\min} \leq 1 - \eta_t \sigma_\rel$, which yields the desired result.

The result on $L_t^\prime$ is be obtained by bounding $\eta / p_{\min}$ by $L_\rel^{\max} = \max_i L_\rel^i$ and remarking that $1 - \eta_t L_\rel^{\max} \leq 1 - \eta_t \sigma_\rel$ since $L_\rel^{\max} \geq \sigma_\rel$.
\end{proof}

\section{Convergence results for DVR} 
\label{app:application_augmented_problem}
We now give a series of small results, that justify our approach. We start by showing the applicability of Theorem~\ref{thm:bregman_coordinate} to Problem~\eqref{eq:dual_problem}, and the associated constants. Finally, we show how to obtain rates for the primal iterates $\theta_t$.

\subsection{Application to the dual of the augmented problem}
In this section, we note $f_{\rm sum}^* = \sum_{i=1}^n \sum_{j=1}^m f_{ij}^*$, so that Problem~\ref{eq:dual_problem} writes:
\begin{equation}
    \min_{x,y} q_A(x,y) + f_{\rm sum}^*(y)
\end{equation}

\begin{lemma} \label{lemma:equivalence_dvr_breg_cd}
The iterations of Algorithm~\ref{algo:DVR} are equivalent to the iteration of Equations~\eqref{eq:coord_breg_constrained} applied to Problem~\eqref{eq:dual_problem} with $g=0$ and $\phi(x,y) = \phi_\comm(x) + \sum_{i=1}^n \sum_{j=1}^m \phi_{ij}(y^{(ij)})$, with $\phi_\comm(x) = \frac{1}{2}\|x\|^2_{A^\dagger A}$ for coordinates associated with communication edges, and $\phi_{ij}(y^{(ij)}) = \frac{L_{ij}}{\mu_{ij}^2}f_{ij}^*(\mu_{ij} y_{ij})$ for coordinates associated with computation edges. 
\end{lemma}
\begin{proof}
This result follows from the dual-free and implementation-friendly derivations presented in the previous section. 
\end{proof}

\begin{lemma} \label{lemma:rel_smooth_sc}
Let $\alpha = 2\lambda_{\min}(A^\top_\comm D_M^{-1} A_\comm)$, and $\phi$ as in Lemma~\ref{lemma:equivalence_dvr_breg_cd}, then:
\begin{enumerate}
    \item $q_A + f_{\rm sum}^*$ is ($\alpha /2$)-strongly convex relatively to $\phi$.
    \item $q_A + f_{\rm sum}^*$ is ($L_\rel^\comm$)-smooth relatively to $\phi$ in the direction of communication edges, with \[L_\rel^\comm = \lambda_{\max}(A_\comm^\top \Sigma_\comm A_\comm).\]
    \item $q_A + f_{\rm sum}^*$ is ($L_\rel^{ij}$)-smooth relatively to $\phi$ in the direction of virtual edge $(i,j)$, with \[L_\rel^{ij} = \alpha\left(1 + \frac{L_{ij}}{\sigma_i}\right).\]
\end{enumerate}
\end{lemma}

\begin{proof}
First note that $\nabla^2 f_{\rm sum}^*$ is a block-diagonal matrix, and its $ij$-th block is equal to 
\begin{equation}\label{eq:sc_f_sum}
    (\nabla^2 f_{\rm sum}^*(y))_{ij} = A^\top (u_{ij} u_{ij}^\top \otimes \nabla^2 f_{ij}^*(\mu_{ij}y^{(ij)})) A \succcurlyeq \frac{1}{L_{ij}} A^\top (u_{ij} u_{ij}^\top \otimes P_{ij}) A,
\end{equation}
where $u_{ij} \in \R^{n(1+m)}$ denotes the unit vector corresponding to virtual \emph{node} $(i,j)$. We denote $\tilde{\Sigma} = \Sigma + \sum_{i=1}^n \sum_{j=1}^m \frac{1}{L_{ij}} (u_{ij} u_{ij}^\top) \otimes I_d$. Then, 
\begin{equation} \label{eq:split_qa_fsum}
    \nabla^2 q_A(x, y) + \nabla^2 f_{\rm sum}^*(y) = A^\top \tilde{\Sigma} A + \nabla^2 f_{\rm sum}^*(y) - A^\top \left[\sum_{i=1}^n \sum_{j=1}^m \frac{1}{L_{ij}} (u_{ij} u_{ij}^\top) \otimes P_{ij}\right] A.
\end{equation}

\paragraph{Relative strong convexity.} Then,~\citet[Lemma 6.5]{hendrikx2020optimal} leads to $A^\top \tilde{\Sigma} A \succcurlyeq \sigma_F A^\dagger A$. Note that the notations are slightly different, and the matrix $\tilde{\Sigma}$ in this paper is the same as the matrix $\Sigma^\dagger$ in~\citet{hendrikx2020optimal}. Then, remark that $(A^\dagger A)_{ij} = P_{ij} = \frac{1}{\mu_{ij}^2} (A^\top [(u_{ij} u_{ij}^\top) \otimes P_{ij}] A)_{ij}$, and $\phi_{ij} = \alpha^{-1} f_{ij}^*$, so that:
\begin{align*}
    \nabla^2 q_A(x, y) + \nabla^2 f_{\rm sum}^*(y) &\succcurlyeq \sigma_F \nabla^2 \phi(x, y) + 
    \\&(1 - \alpha^{-1} \sigma_F)\left[\nabla^2 f_{\rm sum}^*(y) - A^\top \left[\sum_{i=1}^n \sum_{j=1}^m \frac{1}{L_{ij}} (u_{ij} u_{ij}^\top) \otimes P_{ij}\right] A.\right].
\end{align*}
Finally, using that Equation~\eqref{eq:sc_f_sum} along with the fact that $\sigma_F \leq \alpha$ implies that $q_A + f_{\rm sum}^*$ is $\sigma_F$-relatively strongly convex with respect to $\phi$.

\paragraph{Relative smoothness.} We first prove the relative smoothness property for communicate edges. For any $\tilde{x} \in R^{Ed}$, Equation~\eqref{eq:split_qa_fsum} leads to:
\begin{equation*}
    (\tilde{x}, 0)^\top [\nabla^2 q_A(x, y) + \nabla^2 f_{\rm sum}^*(y)](\tilde{x}, 0) = (\tilde{x}, 0)^\top A^\top \Sigma A (\tilde{x}, 0) \preccurlyeq L_\rel^\comm(\tilde{x}, 0)^\top\nabla^2\phi(x,y) (\tilde{x}, 0).
\end{equation*}
Similarly, for any $\theta \in \R^d$, we consider $\tilde{y} = e_{ij} \otimes \theta$ and write: 
\begin{align*}
    \tilde{y}^\top[\nabla^2 q_A(x, y)& + \nabla^2 f_{\rm sum}^*(y)]\tilde{y} = \tilde{y}^\top A^\top \tilde{\Sigma} A \tilde{y} + \mu_{ij}^2 \theta^\top \left[\nabla^2 f_{ij}^*(\mu_{ij}y^{(ij)}) - \frac{1}{L_{ij}}P_{ij}\right] \theta\\
    &\preccurlyeq L_\rel^i \tilde{y}^\top \nabla^2\phi(x,y)\tilde{y} + (1 - \alpha^{-1}L_\rel^i)\theta^\top \left[\nabla^2 f_{ij}^*(\mu_{ij}y^{(ij)}) - \frac{1}{L_{ij}}P_{ij}\right] \theta,
\end{align*}
with
\[L_\rel^i = \max_\theta \mu_{ij}^2 u_{ij}^\top \tilde{\Sigma} u_{ij} \frac{\theta^\top P_{ij} \theta}{\|\theta\|^2} \leq \alpha\left(1 + \frac{L_{ij}}{\sigma_i}\right).\]
Finally, $\nabla^2 f_{ij}^*(\mu_{ij}y^{(ij)}) \succcurlyeq P_{ij} / L_{ij}$, and $\alpha \leq L_\rel^i$, which ends the proof of the directional relative smoothness result.
\end{proof}

\begin{lemma} \label{lemma:sep_qA_plus_f}
Assumption~\ref{ass:separability} holds with $f = q_A + f^*_{\rm sum}$, $g=0$, and $\phi$ as in Lemma~\ref{lemma:equivalence_dvr_breg_cd}, and when the sampling is such that either:
\begin{itemize}
    \item All communication edges are sampled at once, or
    \item Each node samples exactly one virtual edge. 
\end{itemize}
\end{lemma}

\begin{proof}
First of all, $g=0$ is separable, and $\phi$ is separable with respect to the communication and computation blocks by construction.

We note $b_\comm$ the block of all communication edges, which is sampled with probability $p_\comm$. All communication edges are sampled at the same time, so $p_i = p_\comm$ for all $i \in b_\comm$ and so $\phi$ respects option $2$ for the communication block.

Let us now consider a computation block $b$. First of all, $\phi$ is separable for the virtual edges. Then, virtual blocks contain exactly one virtual edge per node, and so $b = \{(1, j_1), \cdots, (n, j_n)\}$. Let $k \neq \ell$, then
\begin{align*}
    e_{k, j_k}^\top A^\top \Sigma  A e_{\ell, j_\ell} = \mu_{k, j_k} \mu_{\ell, j_\ell} (e_k - e_{k, j_k})^\top \Sigma  (e_\ell - e_{\ell, j_\ell}) = 0.
\end{align*}
Therefore,
\begin{align*}
q_A\left(x_t + \sum_{(i,j) \in b} \delta_{ij}\right) - q_A(x_t) &= \frac{1}{2}\left(\sum_{(i,j) \in b} \delta_{ij}\right) A^\top \Sigma  A \Bigg(\sum_{(ij) \in b} \delta_{ij}\Bigg) + \left(\sum_{(ij) \in b} A\delta_{ij} \right)^\top \Sigma  A x_t\\
&= \sum_{(i,j) \in b} \left(q_A(\delta_{ij}) + \delta_{ij}^\top A^\top \Sigma  A x_t\right)\\
&= \sum_{(i,j) \in b}\left(q_A(x_t + \delta_{ij}) - q_A(x_t)\right).
\end{align*}
Finally, $f_{\rm sum}^*$ is separable, and so $q_A + f_{\rm sum}^*$ respects option 2.
\end{proof}

We can now prove the main theorem on the convergence rate of DVR.

\begin{theorem}\label{thm:dvr_dual_guarantees}
We choose $p_\comm = \big(1 + \gamma \frac{m + \kappa_s}{\kappa_\comm}\big)^{-1}$ and $p_{ij} \propto (1 - p_\comm)(1 + L_{ij} / \sigma_i)$. Then, for all $\theta_0 \in \R^{n \times d}$ and all $t > 0$, the error is such that:
\begin{equation}
    \frac{\eta_t}{p_{\min}}D_\phi(\lambda_\star, \lambda_t) + D(\lambda_t) - D(\lambda_\star) \leq \left(1 - \frac{\alpha \eta_t}{2} \right)^t \left[\frac{\eta_t}{p_{\min}}D_\phi(\lambda_\star, \lambda_0) + D(\lambda_0) - D(\lambda_\star)\right],
\end{equation}
with $p_{\min} = \min (p_\comm, \min_{ij} p_{ij})$, $\lambda_t = (x_t, y_t)$ and $D = - (q_A + f_{\rm sum}^*)$. Therefore, the expected time $T_\varepsilon$ required to reach precision $\varepsilon$ is equal to:
\[ T_\varepsilon = O\left(\left[2(m + \kappa_s) + \tau \frac{\kappa_\comm}{\gamma}\right]\log\varepsilon^{-1}\right).\]
\end{theorem}

\begin{proof}
Using Lemmas~\ref{lemma:sep_qA_plus_f} and~\ref{lemma:rel_smooth_sc}, we apply Theorem~\ref{thm:bregman_coordinate} (convergence of Bregman coordinate gradient descent), and obtain that the convergence rate is $\eta_t \alpha / 2$, with $\eta_t \leq \min_{ij} p_{ij} / L_\rel^i$ and $\eta_t \leq p_\comm / L_\rel^\comm$. Therefore, for communication edges, we have that 
\[ \eta_t \leq \frac{p_\comm}{L_\rel^\comm} = \frac{p_\comm}{\lambda_{\max}(A_\comm^\top \Sigma_\comm^{-1} A_\comm)}.\]
For computation edges, we know that $p_{ij} = p_\comm (1 + L_{ij} / \sigma_i) / (\sum_{j=1}^m (1 + L_{ij} / \sigma_i))$, and so \[\eta_t \leq \frac{p_{ij}}{L_\rel^{ij}} = \frac{p_\comp}{\alpha \sum_{j=1}^m (1 + \sigma_i^{-1} L_{ij})} \leq \frac{p_\comp}{\alpha (m + \kappa_s)},\]
with $\kappa_s \geq \sigma_i^{-1} \sum_{j=1}^m L_{ij}$ for all $i$.

In the end, we would like these two bounds to be equal, so we choose $p_\comp$ and $p_\comm$ such that $$p_\comp = p_\comm \left(m + \kappa_s\right)\frac{\lambda_{\min}^+(A_\comm^\top D_M^{-1} A_\comm)}{\lambda_{\max}(A_\comm^\top \Sigma_\comm^{-1} A_\comm)}.$$
Yet, we also know that $p_\comm = 1 - p_\comp$, so \[p_\comp = \left(1 + \frac{1}{m + \kappa_s}\frac{\lambda_{\max}(A_\comm^\top \Sigma_\comm^{-1} A_\comm)}{\lambda_{\min}^+(A_\comm^\top D_M^{-1} A_\comm)}\right)^{-1}.\]
Equivalently, this corresponds to taking
\[p_\comm = \left(1 + \gamma \frac{m + \kappa_s}{\kappa_\comm}\right)^{-1}.\]

With this choice, one can verify that $\eta_t$ verifies both $\eta_t \alpha \leq 2 p_\comm$ and $\eta_t \alpha \leq  2\min_{ij}p_{ij}$, so the rate is:
\[1 - \frac{\eta_t \alpha}{2} = 1 - \frac{p_\comp}{2(m + \kappa_s)}.\]
The expected execution time to reach precision $\varepsilon$, denoted $T_\varepsilon$, is equal to $T_\varepsilon = \rho^{-1}(p_\comp + \tau p_\comm) K_\varepsilon$ with $K_\varepsilon$ such that $C(1 - \eta_t \alpha / 2)^{K_\varepsilon} < \varepsilon$ for some constant $C$, and so:
\[T_\varepsilon = O\left(2(m + \kappa_s) + \tau \frac{\kappa_\comm}{\gamma}\right).\]
\end{proof}

\subsection{Primal guarantees}
The goal of this section is to recover primal guarantees from dual guarantees. Although the initial setting is inspired from~\citet{lin2015accelerated}, the proof is different, and in particular does not require smoothness of the $f_{ij}^*$ or an extra proximal step. We define for $\beta \geq 0$ the Lagrangian function:
\begin{equation} \label{eq:def_lagrangian}
    \cL(\lambda, \theta) = \sum_{i=1}^n \sum_{j=1}^m f_{ij}(\theta^{(ij)}) + \frac{\sigma_i}{2}\|\theta^{(i)}\|^2 + \frac{\beta}{2}\|\theta^{(i)} - \omega^{(i)}\|^2 - \lambda^\top A^\top \theta.
\end{equation}
The dual problem $D(\lambda)$ is defined as \[D(\lambda) = \min_\theta \cL(\lambda, \theta).\]
Given an approximate dual solution $\lambda_k$, we can get an approximate primal solution $\theta_k = \arg \min_\theta \cL(\lambda_k, \theta)$, which is obtained as:
\begin{align}
\label{eq:theta_lambda_ij}    &\theta^{(ij)}_t = \arg \min_v \left(f_{ij}(v) - \mu_{ij} \lambda^{(ij)}_t v\right) \in \partial f_{ij}^*(\mu_{ij} \lambda^{(ij)}_t),\\
\label{eq:theta_lambda_i}    & \theta_t^{(i)} = \frac{1}{\sigma_i + \beta}\left((A \lambda_t)^{(i)} + \beta \omega^{(i)}\right).
\end{align}
Note that $\theta_t^{(ij)}$ corresponds to the $z_t^{(ij)}$ from Algorithm~\ref{algo:DVR}. We chose to use a different notation in the main text to emphasize on the fact that these are the parameters for the virtual nodes, but $z_t^{(ij)}$ actually converge to the solution as well. Similarly, $\lambda_t$ corresponds to $(x_t, y_t)$ the concatenation of the parameters for communication and virtual edges from Section~\ref{sec:algo_design}. The last difference is that the Lagrangian defined in Equation~\eqref{eq:def_lagrangian} actually corresponds to a Lagrangian associated to a perturbed version of Problem~\eqref{eq:general_problem} in which $\tilde{f}_i(\theta) = f_i(\theta) + \frac{\beta}{2}\| \theta - \omega^{(i)}\|^2$. The solution to the initial problem can be retrieved by taking $\beta = 0$, but this more general formulation enables us to derive results that also holds for the inner problems solved by the Catalyst accelerated version of DVR.

\begin{lemma} \label{lemma:primal_error}
Denote $C_0 = \frac{(\beta + \sigma_{\max} + L_{\max})}{2(\sigma_{\min} + \beta)^2} \left(\frac{p_{\min}}{\eta_t}D_\phi(\lambda^\star, \lambda_0) + \left(D(\lambda^\star) - D(\lambda_0)\right)\right)$, then
\begin{equation}
\sum_{i=1}^n \|\theta_t^{(i)} - \theta^\star\|^2 \leq  C_0 (1 - \rho)^t.
\end{equation}
\end{lemma}
\begin{proof}
Using the fact that $\theta_t^{(i)} = \frac{1}{\sigma_i + \beta}( (A\lambda_t)^{(i)} + \omega_t^{(i)}$) (and similarly for $\theta^\star$), where $\Sigma_\beta$ is the block diagonal matrix such that $(\Sigma_\beta)_{ii} = (\sigma_i + \beta)^{-1}I_d$, we obtain:
\begin{align*}
    \sum_{i=1}^n \|\theta_t^{(i)} - \theta^\star\|^2 &= \sum_{i=1}^n \frac{1}{(\sigma_i + \beta)^2}\|(A\lambda_t)^{(i)} - (A \lambda^\star)^{(i)}\|^2\\
    &\leq \frac{1}{(\sigma_{\min} + \beta)^2}\|A \lambda_t - A\lambda^\star \|^2.
\end{align*}
Using the $\min((\sigma_{\max} + \beta)^{-1}, L_{ij}^{-1})$-strong convexity of $\theta \mapsto \frac{1}{2}x^\top \Sigma_\beta x + \sum_{i,j} f_{ij}^*(x^{(ij)})$, we obtain:
\begin{equation}
    \sum_{i=1}^n \|\theta_t^{(i)} - \theta^\star\|^2 \leq 2\frac{(\beta + \sigma_{\max} + L_{\max})}{(\sigma_{\min} + \beta)^2} \left(D(\lambda^\star) - D(\lambda_t)\right).
\end{equation}
Then, we add $p_{\min} \eta_t^{-1} D_\phi(\lambda^\star, \lambda_t) \geq 0$ and apply Theorem~\ref{thm:dvr_dual_guarantees}, which yields
\begin{align*}
   \sum_{i=1}^n \|\theta_t^{(i)} - \theta^\star\|^2 \leq \frac{2(\beta + \sigma_{\max} + L_{\max})}{(\sigma_{\min} + \beta)^2}  (1 - \rho)^t \left(\frac{p_{\min}}{\eta_t}D_\phi(\lambda^\star, \lambda_0) + \left(D(\lambda^\star) - D(\lambda_0)\right)\right).
\end{align*}
\end{proof}

Then, Theorem~\ref{thm:dvr_rate} is a direct consequence of Theorem~\ref{thm:dvr_dual_guarantees} and Lemma~\ref{lemma:primal_error}.

\section{Catalyst acceleration}
\label{app:catalyst}
We show in this Section how to apply Catalyst acceleration to DVR, and prove the convergence speed in this case. 

\subsection{Derivation and rates}
In the main text, we derived DVR to solve regularized finite sum problems. Although not so different, the subproblem obtained with Catalyst is not in the form of Problem~\eqref{eq:general_problem}, and some adjustments need to be made. More specifically, we would like to solve problems of the form:
\begin{equation}\label{eq:primal_catalyst}
    \min_\theta 
    \left\{ F_t(\theta) \triangleq \sum_{i=1}^n \left[\frac{\sigma_i}{2}\|\theta\|^2 + \frac{\beta}{2}\|\theta - \omega_t^{(i)}\|^2 + \sum_{j=1}^m f_{ij}(\theta)\right]\right\}. 
\end{equation}
An easy way to adapt the algorithm is to consider the extra $(\beta / 2)\|\theta - \omega_t^{(i)}\|^2$ as just another component of the sum. Yet, the point of this extra term is to make the problem easier to solve by adding strong convexity. This would not be the case if this term were is treated as just another term in the sum. Therefore, we want to include it with the quadratic term. We define:
\[h(x) = \frac{\sigma_i}{2}\|\theta\|^2 + \frac{\beta}{2}\|\theta - \omega_t^{(i)}\|^2,\]
then $h^*(x) = \frac{1}{2(\beta + \sigma)}\|x + \beta \omega_t^{(i)}\|^2 - \frac{\beta}{2}\|\omega_t^{(i)}\|^2$. Therefore, Problem~\eqref{eq:dual_problem} becomes:
\begin{equation}\label{eq:dual_problem_catalyst}
    \min_{\lambda \in \R^{(E + mn)d}} \frac{1}{2} \lambda^\top A^\top \Sigma_\beta A \lambda + \beta \omega_t^\top \Sigma_\beta A \lambda + \sum_{i=1}^n \sum_{j=1}^m f_{ij}^*((A\lambda)_{ij}),
\end{equation}
with $(\Sigma_\beta)_{ii} = (\sigma_i + \beta)^{-1}$ for $i \in \{1, \dots, n\}$. The linear term does not affect the Hessians, and thus the convergence rate is the same as before, with $\sigma$ replaced by $\sigma + \beta$. In terms of algorithms, we just need to modify the gradient term, and obtain Algorithm~\ref{algo:DVR_cata_acc}. The only term that changes is $\nabla q_A(x, y)$, to which an extra $\beta \Sigma_\beta \omega_t$ term is added. Therefore, the updates to $\theta_t$ and $z_t$ remain unchanged, and only the initial expression of $\theta_t$ requires some adjustments since we now have that (as written in Equation~\ref{eq:theta_lambda_i}): 
\[ \theta_{t, k}^{(i)} = \frac{1}{\sigma_i + \beta}\left((A \lambda_{t,k})^{(i)} + \beta \omega^{(i)}\right).\] 

If we only consider 1 inner loop then the only thing that changes is the initial condition. If we consider several outer loops, then the we must choose the new parameter as $\theta_0^{t+1} = \theta_{T}^t + \Sigma_\beta(\omega_{t+1} - \omega_t)$ in order to maintain the invariant, but a remarkable fact is that the inner iterations remain the same, with the only exception that $\Sigma$ is replaced by $\Sigma_\beta$. Note that it is possible to warm-start the $z_{t+1,0}$ as well, but this requires updating $\theta_{t,0}$ accordingly with $\nabla f_{ij}(z_{t,0}^{(ij)})$, which requires a full pass over the local dataset. We therefore choose not to do it.  

\begin{algorithm}
\caption{Accelerated DVR$(z_0)$}
\label{algo:DVR_cata_acc}
\begin{algorithmic}[1]
\STATE $\alpha = 2\lambda_{\min}^+(A_\comm^\top D_M^{-1} A_\comm)$, $\eta = \min\left( \frac{p_\comm}{\lambda_{\max}(A_\comm^\top \Sigma_{\beta, \comm} A_\comm)}, \frac{p_{ij}}{\alpha(1 + \sigma_i^{-1}L_{ij})} \right)$
\STATE $q = \frac{\sigma_{\min}}{\sigma_{\min} + \beta}$ \COMMENT{Initialization}
\STATE $\omega_0^{(i)} = - \frac{1}{\sigma_i + \beta}\sum_{j=1}^m \nabla f_{ij}(z_0^{(ij)})$, $\theta_0^{(i)} = \left(1 + \frac{\beta}{\sigma_i + \beta}\right)\omega_0^{(i)}$. \COMMENT{$z_0$ is arbitrary but not $\theta_0$.}
\FOR[T outer loops]{$t=0$ to $T-1$}
\FOR[Inner loop runs for $K$ iterations]{$k=0$ to $K-1$}
\STATE $z_{t, k+1} = z_{t,k}$.
\STATE Sample $u_t$ uniformly in $[0, 1]$. \COMMENT{Randomly decide the kind of update}
\IF{$u_t \leq p_\comm$}
\STATE $\theta_{t, k+1} = \theta_{t,k} - \frac{\eta_t}{p_\comm} \Sigma_\beta  W \theta_{t,k}$ \COMMENT{Communication using $W$}
\ELSE 
\FOR{$i = 1$ to $n$}
\STATE Sample $j \in \{1, \cdots, m\}$ with probability $p_{ij}$.
\STATE $z_{t, k+1}^{(ij)} = \left(1 - \frac{\alpha \eta}{p_\comp}\right)z_{t,k}^{(ij)} + \frac{\alpha\eta}{p_\comp} \theta_{t,k}^{(i)}$ \COMMENT{Computing new virtual node parameter}
\STATE $\theta_{t,k+1}^{(i)} = \theta_{t,k}^{(i)} - \frac{1}{\sigma_i + \beta}\left(\nabla f_{ij}(z_{t, k+1}^{(ij)}) - \nabla f_{ij}(z_{t,k}^{(ij)})\right)$ \COMMENT{Local update using $f_{ij}$}
\ENDFOR
\ENDIF
\ENDFOR
\STATE $\omega_{t+1} = \theta_{t, K} + \frac{1 - \sqrt{q}}{1 + \sqrt{q}}(\theta_{t, K} - \theta_{t-1,K})$
\STATE $\theta_{t+1, 0} = \theta_{t, K} + \frac{\beta}{\beta + \sigma_i}(\omega_{t+1} - \omega_t)$
\STATE $z_{t+1, 0} = z_{t, K}$
\ENDFOR
\STATE \textbf{return} $\theta_T$
\end{algorithmic}
\end{algorithm}

However, it is not obvious that Algorithm~\ref{algo:DVR_cata_acc} corresponds to a genuine Catalyst acceleration yet. Indeed, Catalyst acceleration requires having a feasible $\varepsilon_t$-approximations for the primal problem, \emph{i.e.}, points $\theta_t \in \R^d$ such that $F_t(\theta_t) - \min_\theta F(\theta) \leq \varepsilon_t$. In our case, we only have dual guarantees and approximate feasibility. We know that the parameters converge to consensus, but they do not reach it at any time. This is a problem because it is then not possible to adequately define $F_{t+1}$ based on the local approximations of the solutions of $F_t$. Yet, following the approach of~\citet{li2020revisiting}, we note that 
\[ \sum_{i=1}^n \|\theta - \omega_t^{(i)}\|^2 = n \| \theta - \bar{\omega}_t\|^2 + \sum_{i=1}^n \|\omega_t^{(i)}\|^2 - n\| \bar{\omega}_t\|^2,\]
where $\bar{\omega}_t = \frac{1}{n} \sum_{i=1}^n \omega_t^{(i)}$. This means that although $F_t$ is only defined with the local variables $\omega_t^{(i)}$, \emph{solving $F_t$ is equivalent to solving a problem involving $\bar{\omega}_t$ only}. Besides, the Catalyst iterations are linear, meaning that performing the extrapolation step on $\bar{\theta}_t$ is equivalent to performing it on each $\theta_t^{(i)}$ individually. Therefore, although Catalyst is implemented in a fully decentralized manner (each node knowing only its own parameter), it is conceptually applied to a mean parameter $\bar{\theta}_t$ (that is never explicitly computed). In the following, we thus analyze the performances of the following algorithm:
\begin{align}
    \begin{split} \label{eq:catalyst_average_formulation}
        &\bar{\theta}_{t+1} \approx \arg \min_\theta F(\theta) + \frac{n\beta}{2}\|\theta - \bar{\omega}_t\|^2\\
        &\bar{\omega}_t = \bar{\theta}_{t+1} + \frac{1 - \sqrt{q}}{1 + \sqrt{q}}(\bar{\theta}_{t+1} - \bar{\theta}_t),
    \end{split}
\end{align}
where we recall that $q = \sigma_{\min}/(\sigma_{\min} + \beta)$. Recall that the inner problem is approximated using DVR and the means do not need to be computed explicitly. Let $\kappa_s^\beta = \max_i 1 + (\sum_{j=1}^m L_{ij}) / (\beta + \sigma_i)$, and $\kappa_\comm^\beta$ be obtained similarly to $\kappa_\comm$ but replacing $\Sigma$ by $\Sigma_\beta$. We consider in this section that $\sigma_i = \sigma$ for all $i \in \{1, \dots, n\}$ in order to simplify exposition, but the results hold more generally. Note that $\alpha$ and $\eta$ have slightly different expressions than in the main text since $\beta$ is now involved in their definitions. We define the sequence $\varepsilon_t$ which is such that:
\begin{equation}\label{eq:def_epsilon_t}
    \varepsilon_t = \frac{2}{9}\left(F(\theta_0) - F(\theta^\star)\right)(1 - \rho^\out)^t \hbox{ with } \rho^\out < \sqrt{q}, \hbox{ and } q = \frac{\sigma}{\sigma + \beta}.
\end{equation}
We then prove the following theorem:
\begin{theorem} \label{thm:rate_dvr_catalyst_precise}
Consider Algorithm~\ref{algo:DVR_cata_acc} with $p_\comm = \big(1 + \gamma \frac{m + \kappa_s^\beta }{\kappa_\comm^\beta}\big)^{-1}$, $p_{ij} \propto (1 - p_\comm)(1 + L_{ij} / (\sigma_i + \beta))$. If $K = \tilde{O}\left(1 / (\eta_t \alpha)\right)$ then for all $t \leq T$, $F_t(\bar{\theta}_t) - F_t(\theta_t^\star) \leq \varepsilon_t$ and
\begin{equation}
    F(\bar{\theta}_t) - F(\theta^\star) \leq \frac{8}{(\sqrt{q} - \rho^\out)^2}(1 - \rho^\out)^{t + 1}(F(\bar{\theta}_0) - F(\theta^\star)).
\end{equation}
\end{theorem}

Note that the error is on the mean parameter, and we also want $\theta_t^{(i)}$ to be close to $\bar{\theta}_t$ for all $i$. This is ensured by Lemma~\ref{lemma:primal_error}. Before we start the proof of Theorem~\ref{thm:rate_dvr_catalyst_precise}, we show that Theorem~\ref{thm:rate_dvr_catalyst} is a corollary of Theorem~\ref{thm:rate_dvr_catalyst_precise}.

\begin{proof}[Proof of Theorem~\ref{thm:rate_dvr_catalyst}]
Using the same argument as in Theorem~\ref{thm:dvr_rate}, we obtain that each inner loop takes time \[T_{\rm inner} = O\left(m + \frac{L_s + \sigma}{\beta + \sigma} + \tau \frac{L_\comm + \beta}{\gamma(\beta + \sigma)}\right)\] 
in expectation, so the total number of inner iterations is of order:
\begin{equation}
    T_\varepsilon = \tilde{O}\left(\sum_{k=0}^{\lceil 1 / \rho^{\rm out}\rceil}T_{\rm inner}\right) = \tilde{O}\left(\sqrt{1 + \frac{\beta}{\sigma}}\left(m + \frac{L_s + \sigma}{\beta + \sigma} + \tau \frac{L_\comm + \beta}{\gamma(\beta + \sigma)} \right)\log\frac{1}{\varepsilon} \right).
\end{equation}
Therefore, we see that if we choose $\beta + \sigma = L_\comm$ then, taking into account the fact that $\kappa_s \leq m \kappa_\comm$, the algorithm takes time:
\[ T_\varepsilon = \tilde{O}\left(\sqrt{\kappa_\comm}\left(m + \frac{\tau_c}{\gamma}\right) \right).\]
Therefore, using Chebyshev acceleration allows to recover the rate of optimal batch algorithms (up to log factors). On the other hand, if we choose $\beta = L_s / m - \sigma$ then if $\beta \geq 0$ (\emph{i.e.}, $\kappa_s \geq m$), the time to convergence is equal to:
\[ T_\varepsilon = \tilde{O}\left(\sqrt{\frac{\kappa_s}{m}}\left(m + \tau\frac{m\kappa_\comm + \kappa_s}{\gamma \kappa_s}\right) \right).\]
This can be rewritten as:
\[T_\varepsilon = \tilde{O}\left(\sqrt{m\kappa_s} + \tau\frac{\sqrt{\kappa_\comm}}{\gamma} \sqrt{\frac{m\kappa_\comm}{\kappa_s}}\right).\]
Therefore, we obtain the optimal $\sqrt{m\kappa_s}$ computation complexity in this case, with a slightly suboptimal communication complexity due to the $\sqrt{m\kappa_\comm / \kappa_s}$ term. When this term is equal to $1$ then $\sqrt{m\kappa_s} = m\sqrt{\kappa_b}$ and so nothing is gained from using a stochastic algorithm. Otherwise, this allows to trade-off communications for computations. 
\end{proof}

The proof of Theorem~\ref{thm:rate_dvr_catalyst_precise} is obtained in several steps, that we emphasize below:
\begin{enumerate}
    \item Equivalent decentralized implementation of Catalyst.
    \item Bounding the primal suboptimality as $F_t(\bar{\theta}_t) - \min_\theta F_t(\theta) \leq (1 - (\eta \alpha)/2)^k D_0^t$, with $k$ the number of inner iterations and $D_0^t$ a dual error. This quantifies how precisely the inner problem is solved. 
    \item Evaluating the initial dual suboptimality $D_0^t$, which depends on $\theta_{t-1}$ (and its associated dual parameter $\lambda_{t-1}$). This quantifies how good $\bar{\theta}_{t-1}$ already is as a solution to $F_t$.
\end{enumerate}
In the end, this allows us to use the catalyst general results with primal criterion, and with simple warm-start scheme (warm-start on the last iterate of the last outer iteration). The first point is presented at the beginnning of this section and the second one is adressed by Lemma~\ref{lemma:primal_error}. The following section deals the last point.

\subsection{Proof of Theorem~\ref{thm:rate_dvr_catalyst_precise}}
We now show a bound on the initial error of an inner loop when warm-starting on the last iterate of the previous inner loop. Indeed, the convergence results for DVR depend on the initial dual error and so results from~\citep{lin2017catalyst} cannot be used directly. Yet, it can be adapted, as we show in this section. We note $D_t(\lambda)$ the dual function at outer step $t$ (which should not be mistaken with the Bregman divergence $D_\phi$), and $\lambda_\star^t$ its minimizer. Similarly, we note $\theta_\star^t = \arg \min_\theta F_t(\theta)$, whereas $\theta^\star$ is the global minimizer of $F$. The following theorem ensures convergence of $\bar{\theta}_t$ to the true optimum, given that the subproblems are solved precisely enough.
\begin{theorem}
\citep[Proposition 5]{lin2017catalyst}. If $F_k(\bar{\theta}_k) - F_k(\theta_\star^k) \leq \varepsilon_k$ for all $k \leq t$ then
\begin{equation} \label{eq:thm_convergence_cata_outer}
    F(\bar{\theta}_t) - F(\theta^\star) \leq \frac{8}{(\sqrt{q} - \rho^\out)^2}(1 - \rho^\out)^{t + 1}(F(\bar{\theta}_0) - F(\theta^\star)).
\end{equation}

\end{theorem}

Therefore, our goal is to prove that $F_t(\bar{\theta}_{t+1}) - F_t(\theta_\star^t) \leq \varepsilon_t$ for all $t$. The smoothness of $F_t$ ensures that this is achieved if 
\begin{equation} \label{eq:control_theta_comm}
    \sum_{i=1}^n \|\theta_{t+1}^{(i)} - \theta_\star^t\|^2 \leq \frac{n}{L} \varepsilon_t.
\end{equation}
Yet, using Lemma~\ref{lemma:primal_error}, we know that, since $\theta_{t+1}^{(i)}$ is obtained by applying $K$ steps of DVR to $F_t$ starting from $\lambda_0^{t}$. 
\begin{equation*}
    \sum_{i=1}^n \|\theta_{t+1}^{(i)} - \theta_\star^t\|^2 \leq \frac{(\beta + \sigma_{\max} + L_{\max})}{(\sigma_{\min} + \beta)^2}  (1 - \rho)^K \left(\frac{p_{\min}}{\eta_t}D_\phi(\lambda_\star^t, \lambda_0^t) + D_t(\lambda_\star^t) - D_t(\lambda_0^t)\right).
\end{equation*}
Unfortunately, we have no control over the dual error at this point. In the remainder of this section, we prove by recursion that Equation~\eqref{eq:control_theta_comm} holds for all $t$. More specifically, we start by assuming that:
\begin{align}
\label{eq:recursion_theta_comm}    & \frac{1}{2}\sum_{i=1}^n \|\theta_{t+1}^{(i)} - \theta_\star^t\|^2 \leq \frac{n}{L} \varepsilon_t, \\
\label{eq:recursion_theta_comp}    & \frac{1}{2}\sum_{i=1}^n \sum_{j=1}^m \|\theta_{t+1}^{(ij)} - \theta_\star^t\|^2 \leq C_1 \varepsilon_t, \\
\label{eq:recursion_dual_error}    & D_t(\lambda_\star^t) - D_t(\lambda_{t+1}) \leq C_2 \varepsilon_t, 
\end{align}
where $C_1$ and $C_2$ are such that the conditions are verified for $t = -1$, with $D_{-1} = D_0$, $\theta_\star^{-1} = \theta_\star^0$, and $\lambda_\star^{-1} = \lambda_\star^0$. Equation~\eqref{eq:recursion_theta_comm} may not hold for $t = -1$, but making it hold at time $t=0$ would only require a slightly longer first inner iteration, meaning at most an extra $\log$ factor. Therefore we assume without loss of generality that it is the case, since the final complexities are given up to logarithmic factors. The rest of this section is devoted to showing that if $K$ is chosen as in Theorem~\ref{thm:rate_dvr_catalyst_precise} then Equations~\eqref{eq:recursion_theta_comm},~\eqref{eq:recursion_theta_comp} and~\eqref{eq:recursion_dual_error} hold regardless of $t$. The first part focuses on assessing the initial error of outer iteration $t+1$ when the conditions hold at the end of outer iteration $t$, and the second part on showing how these errors shrink during outer iteration $t+1$.

\subsubsection{Warm-start error}
We know that DVR converges linearly, and so the error for each subproblem decreases exponentially fast. Yet, we need to know how big the error is when solving a new problem in order to make sure that the progress from solving previous subproblems is not lost. The point of this is to avoid an extra $\log(\varepsilon^{-1})$ factor in the rate, which would come from having to solve each subproblem from a $O(1)$ precision to an $\varepsilon$ precision using DVR. We show in this section that the initial error is actually much lower than $O(1)$ and decreases with the outer iterations. We first start by bounding the variations of $\omega_t$ across iterations, which we will need for the next proofs.

\begin{lemma}[Distance between subproblems]\label{lemma:control_yt}
It holds that
\[ \|\omega_t - \omega_{t-1}\|^2 \leq C_\omega \varepsilon_{t-1}, \hbox{ with } C_\omega = \frac{1080 n}{1 - \rho^\out}\left( \frac{8(1 - \rho^\out)}{\sigma_{\min}(\sqrt{q} - \rho^\out)^2} + \frac{4}{9L} \right).\]
\end{lemma}

\begin{proof}
The form of the updates yields that (see~\citet[Proposition 12]{lin2017catalyst} or \citet[Proof of Lemma 10]{li2020revisiting})
\[\|\omega_t^{(i)} - \omega_{t-1}^{(i)}\| \leq 40 \max \{\|\theta_t^{(i)} - \theta^\star\|, \|\theta_{t-1}^{(i)} - \theta^\star\|, \|\theta_{t-2}^{(i)} - \theta^\star\|\}.\]
Note that here, $\theta^\star$ is the actual solution of the primal problem without the catalyst perturbation. Then, the error can be decomposed as:
\begin{align*}
    \sum_{i=1}^n\|\theta_t^{(i)} - \theta^\star\|^2 
    &\leq 3\sum_{i=1}^n\left(\|\theta_t^{(i)} - \theta_\star^t\|^2 + \| \theta_\star^t - \bar{\theta}_t\|^2 + \|\bar{\theta_t} - \theta^\star\|^2\right)\\
    &\leq 3n \|\bar{\theta_t} - \theta^\star\|^2 + 6\sum_{i=1}^n\|\theta_t^{(i)} - \theta_\star^t\|^2.
\end{align*}
Finally, the strong convexity of $F$ leads to 
\begin{equation}
    \frac{\sigma_{\min}}{2}\|\bar{\theta}_t - \theta^\star\|^2 \leq F(\bar{\theta}_t) - F(\theta^\star) \leq \frac{8}{(\sqrt{q} - \rho^\out)^2}(1 - \rho^\out)^{t + 1}(F(\theta_0) - F(\theta^\star))
\end{equation}
where in the last inequality we use~\citep[Proposition 5]{lin2017catalyst}, which holds because $F_k(\bar{\theta}_k) - F_k(\theta^k_\star) \leq \varepsilon_k$ for all $k < t$. Indeed, $K$ is such that for all $k \leq t$,  $\frac{1}{2}\sum_{i=1}^n \|\theta_k^{(i)} - \theta^k_\star\|^2 \leq \frac{n}{L} \varepsilon_k$, which yields:
\[F_k(\bar{\theta}_k) - F_k(\theta^k_\star) \leq \frac{L}{2}\|\bar{\theta}_k - \theta_\star^k\|^2 \leq \frac{L}{2n} \sum_{i=1}^n \|\theta_k^{(i)} - \theta^k_\star\|^2 \leq \varepsilon_k.\]

Therefore, 
\begin{align*}
    \sum_{i=1}^n\|\theta_t^{(i)} - \theta^\star\|^2 \leq 6n\left(1 - \rho^\out\right)^t(F(\theta_0) - F(\theta^\star)) \left( \frac{8(1 - \rho^\out)}{\sigma_{\min}(\sqrt{q} - \rho^\out)^2} + \frac{4}{9L} \right),
\end{align*}
and a similar bound can be used for $\theta_{t-1}^{(i)}$ and $\theta_{t-2}^{(i)}$. Then, we finish proof by plugging in the expression of $\varepsilon_{t-1}$. 
\end{proof}

We then use Lemma~\ref{lemma:control_yt} to bound the initial dual error. We denote $\theta_k^t$ (and $\lambda_k^t$) the parameters at inner iteration $k$ of outer iteration $t$.  

\begin{lemma}[Dual error warm-start] \label{lemma:warm_start_dual}
The warm-started dual error verifies:
\begin{equation}\label{eq:warm_start_dual}
     D_t(\lambda_\star^t) - D_t(\lambda_t) \leq C_D \varepsilon_{t-1}, \hbox{ with } C_D = \left(C_2 + C_\omega + 4\frac{\beta n}{L}\right).
\end{equation}
\end{lemma}

Note that we simply warm-start the dual coordinates for an outer iteration using the last iterate from the previous one. Yet, this leads to $\theta_0^t = \theta_K^{t-1} + \beta\Sigma_\beta^{-1}(\omega_{t+1} - \omega_t)$, as in Algorithm~\ref{algo:DVR_cata_acc}.

\begin{proof}
Equation~\eqref{eq:dual_problem_catalyst} implies that $D_t(\lambda)$ can be written as:
\begin{equation}
D_t(\lambda) = - \sum_{i=1}^n \frac{1}{\beta + \sigma_i}\left[\frac{1}{2}(A\lambda)^{(i)} + \beta \omega_t^{(i)} \right]^\top(A\lambda)^{(i)} + R_{\comp}(\lambda),
\end{equation}
with $R_\comp(\lambda)$ that only depends on $\lambda^{(ij)}$ and not on $\omega_t^{(i)}$ for $i \in \{1, \cdots, n\}$. Therefore, 
\begin{align*}
    &D_t(\lambda_\star^t) - D_t(\lambda_K^{t-1})\\
    &= D_{t-1}(\lambda_\star^t) - D_{t-1}(\lambda_K^{t-1}) - \beta \sum_{i=1}^n \left[(A\lambda_\star^t)^{(i)} - (A \lambda_K^{t-1})^{(i)}\right]^\top \Sigma _\beta\left[\omega_{t}^{(i)} - \omega_{t-1}^{(i)}\right].
\end{align*}
Equation~\eqref{eq:theta_lambda_i} writes $(A \lambda_\star^{t})^{(i)} = (\beta + \sigma_i) \theta_\star^{t}  - \beta \omega_t^{(i)}$, and so:
\[A\lambda_\star^t - A \lambda_K^{t-1} = A\lambda_\star^t - A\lambda_\star^{t-1} + A \lambda_\star^{t-1} -  A \lambda_K^{t-1} = \Sigma_\beta^{-1} (\theta_\star^t - \theta_\star^{t-1}) + A\lambda_\star^{t-1} - A\lambda_K^{t-1} - \beta(\omega_t - \omega_{t-1}).\]
Then, we know from the equivalent reformulation of Equation~\eqref{eq:catalyst_average_formulation} that $\theta_t^\star = \arg \min F(\theta) + \frac{\beta}{2}\|\theta - \bar{\omega}_t\|^2$, so using the $1$-Lipschitzness of the proximal operator yields 
\begin{equation}\label{eq:diff_theta_star}
    \|\theta^{t}_\star - \theta^{t-1}_\star\|^2 \leq \|\bar{\omega}_t - \bar{\omega}_{t-1}\|^2 \leq \frac{1}{n}\sum_{k=1}^n \|\omega_t^{(k)} - \omega_{t-1}^{(k)}\|^2 = \frac{1}{n}\|\omega_t - \omega_{t-1}\|^2.
\end{equation}
Similarly, $\Sigma_\beta(A\lambda_\star^{t-1} - A\lambda_K^{t-1}) = \theta_\star^{t-1} - (\theta_K^{t-1})^{(i)}$, and so:
\begin{align*}
 \sum_{i=1}^n &\left[(A\lambda_\star^t)^{(i)} - (A \lambda_K^{t-1})^{(i)}\right] \Sigma _\beta\left[\omega_{t}^{(i)} - \omega_{t-1}^{(i)}\right] \leq \sum_{i=1}^n\left\|\frac{(A\lambda_\star^t)^{(i)} - (A \lambda_K^{t-1})^{(i)}}{\beta + \sigma_i}\right\| \left\|\omega_{t}^{(i)} - \omega_{t-1}^{(i)}\right\| \\ 
 &\sum_{i=1}^n 2 \|\theta_\star^t - \theta_\star^{t-1}\|^2 + 2\|\theta_\star^{t-1} - (\theta_K^{t-1})^{(i)} \|^2 + \left(\frac{\beta}{\beta + \sigma_i} + 4\right)\|\omega_t^{(i)} - \omega_{t-1}^{(i)}\|^2.
\end{align*}
Plugging in Equation~\eqref{eq:diff_theta_star} yields:
\begin{equation*}
     D_t(\lambda_\star^t) - D_t(\lambda_K^{t-1}) \leq D_{t-1}(\lambda_\star^t) - D_{t-1}(\lambda_K^{t-1}) + 2 \beta \sum_{i=1}^n\| (\theta_K^{t-1})^{(i)} - \theta_\star^{t-1} \|^2 + 7\beta \|\omega_t - \omega_{t-1}\|^2
\end{equation*}
Finally note that $D_{t-1}(\lambda^{t}_\star) \leq D_{t-1}(\lambda^{t-1}_\star)$ since $\lambda^{t-1}_\star$ is the maximizer of $D_{t-1}$, and $(\theta_K^{t-1})^{(i)} = \theta_t^{(i)}$ since it is the output of DVR after inner iteration $t$. The final expression is obtained using~\ref{lemma:control_yt} and the recursion assumptions given by Equations~\eqref{eq:recursion_theta_comm} and~\eqref{eq:recursion_dual_error}.
\end{proof}

Finally, the warm-start error on the nodes parameters is given by the two following lemmas.
\begin{lemma}[Virtual parameters warm-starts]\label{lemma:warm_start_theta_comp}
Denote $\|\theta_1 - \theta_2\|^2_\comp = \sum_{i=1}^n \sum_{j=1}^m \|\theta_1^{(ij)} - \theta_2^{(ij)}\|^2$. Then, 
\begin{equation}
    \|\theta_0^t - \theta_\star^t\|^2_\comp \leq 2(C_\omega + 2mC_1) \varepsilon_{t-1}.
\end{equation}
\end{lemma}

\begin{proof}
We use the fact that $(\theta_t)^{(ij)} = (\theta_0^t)^{(ij)} = (\theta_K^{t-1})^{(ij)}$ to write:
\begin{align*}
    \|\theta_0^t - \theta_\star^t\|^2_\comp &= \|\theta_K^{t-1} - \theta_\star^{t-1} + \theta_\star^{t-1} - \theta_\star^{t}\|^2_\comp \leq 2\|\theta_t - \theta_\star^{t-1}\|^2_\comp + 2nm\|\theta_\star^{t-1} - \theta_\star^{t}\|^2.
\end{align*}
Then, as before, the 1-Lipchitzness of the prox operator yields $\|\theta_\star^{t-1} - \theta_\star^{t}\| \leq \frac{1}{n}\|\omega_t - \omega_{t-1}\|$.
\end{proof}

\begin{lemma}[Parameters warm-start] \label{lemma:warm_start_theta_comm}
Denote $\|\theta_1 - \theta_2\|^2_\comp = \sum_{i=1}^n \sum_{j=1}^m \|\theta_1^{(ij)} - \theta_2^{(ij)}\|^2$. Then, 
\begin{equation}
    \sum_{i=1}^n \|(\theta_0^t)^{(i)} - \theta_\star^t\|^2 \leq 6\left(C_\omega +  \frac{n}{L}\right)\varepsilon_{t-1}.
\end{equation}
\end{lemma}

\begin{proof}
We use the fact that since $\lambda_0^t = \lambda_K^{t-1}$ then $(\theta_0^t)^{(i)} = (\theta_0^t)^{(i)} + \frac{\beta}{\beta + \sigma_i}(\omega_t^{(i)} - \omega_{t-1}^{(i)})$ to write:
\begin{align*}
    \sum_{i=1}^n \|(\theta_0^t)^{(i)} - \theta_\star^t\|^2 &\leq \sum_{i=1}^n\|(\theta_K^{t-1})^{(i)} - \theta_\star^{t-1} + \theta_\star^{t-1} - \theta_\star^t + \frac{\beta}{\sigma_i + \beta}(\omega_t^{(i)} - \omega_{t-1}^{(i)})\|^2\\
    &\leq 3\|\omega_t - \omega_{t-1}\|^2 + 3n \|\theta_\star^{t-1} - \theta_\star^t\|^2 + 3\sum_{i=1}^n\|(\theta_t)^{(i)} - \theta_\star^{t-1}\|^2.
\end{align*}
\end{proof}

We finish this part on warm starts by proving the following lemma, that links the initial dual parameters error (computed with the Bregman divergence of $\phi$), to the other parameters which we already know how to control. 
\begin{lemma}[Dual parameters warm-start, as measured by the Bregman divergence] \label{lemma:warm_start_breg}
    \begin{equation}
        D_\phi(\lambda_\star^t, \lambda_0^t) \leq C_\phi \varepsilon_{t-1},
    \end{equation}
    with $C_\phi = \frac{6\left(C_\omega + n / L\right) + L_{\max}^2(2C_\omega + 2mC_1)}{\lambda_{\min}^+(A^\top \Sigma_\beta^2 A)}  + \frac{2L_{\max}(C_\omega + 2mC_1)}{\alpha}$.
\end{lemma}

\begin{proof}
We first decompose the Bregman divergence as:
\begin{equation} \label{eq:warm_start_breg_main}
     D_\phi(\lambda_0^t, \lambda_\star^t) \leq \frac{1}{2}\|(\lambda_0^t)^\comm - (\lambda_\star^t)^\comm\|_{A_\comm^\dagger A_\comm} + \sum_{i=1}^n \sum_{j=1}^m D_{\phi_{ij}}((\lambda_\star^t)^{(ij)}, (\lambda_0^t)^{(ij)}). 
\end{equation}
Then, we bound the communication term as:
\begin{align*}
    \|(\lambda_0^t&)^\comm - (\lambda_\star^t)^\comm\|_{A_\comm^\dagger A_\comm} \leq \| \lambda_0^t - \lambda_\star^t\|_{A^\dagger A} \leq \frac{1}{\lambda_{\min}^+(A^\top \Sigma_\beta^2 A)}\|\Sigma_\beta A\left(\lambda_0^t - \lambda_\star^t\right)\|^2\\
    &= \frac{1}{\lambda_{\min}^+(A^\top \Sigma_\beta^2 A)} \left(\sum_{i=1}^n \|(\theta_0^t)^{(i)} - \theta_\star^t\|^2 + \sum_{i=1}^n\sum_{j=1}^m \mu_{ij}^2 \|(\lambda_0^t)^{(ij)} - (\lambda_\star^t)^{(ij)}\|^2 \right)\\
    &= \frac{1}{\lambda_{\min}^+(A^\top \Sigma_\beta^2 A)} \left(\sum_{i=1}^n \|(\theta_0^t)^{(i)} - \theta_\star^t\|^2 + \sum_{i=1}^n\sum_{j=1}^m \|\nabla f_{ij}((\theta_0^t)^{(ij)}) - \nabla f_{ij} ((\theta_\star^t)^{(ij)})\|^2 \right)\\
    &= \frac{1}{\lambda_{\min}^+(A^\top \Sigma_\beta^2 A)} \left(\sum_{i=1}^n \|(\theta_0^t)^{(i)} - \theta_\star^t\|^2 + \sum_{i=1}^n\sum_{j=1}^m L_{ij}^2 \|(\theta_0^t)^{(ij)} - (\theta_\star^t)^{(ij)}\|^2 \right).
\end{align*}
Using Lemmas~\ref{lemma:warm_start_theta_comp} and~\ref{lemma:warm_start_theta_comm}, we obtain: 
\begin{equation}\label{eq:warm_start_breg_comm}
    \frac{1}{2}\|(\lambda_0^t)^\comm - (\lambda_\star^t)^\comm\|_{A_\comm^\dagger A_\comm} \leq \frac{\left(6\left(C_\omega + \frac{n}{L}\right) + L_{\max}^2(2C_\omega + 2mC_1)\right)}{\lambda_{\min}^+(A^\top \Sigma_\beta^2 A)} \varepsilon_{t-1}.
\end{equation}

For the computation part, we use the duality property of the Bregman divergence, which yields
\begin{align*}
    D_{\phi_{ij}}((\lambda_\star^t)^{(ij)}, (\lambda_0^t)^{(ij)}) &= \frac{L_{ij}}{\mu_{ij}^2} D_{f_{ij}^*}(\mu_{ij}(\lambda_\star^t)^{(ij)}, \mu_{ij}(\lambda_0^t)^{(ij)})\\
    &= \frac{L_{ij}}{\mu_{ij}^2} D_{f_{ij}}(\nabla f_{ij}(\mu_{ij}(\lambda_0^t)^{(ij)}), \nabla f_{ij}(\mu_{ij}(\lambda_\star^t)^{(ij)}))\\
    &= \frac{L_{ij}}{\mu_{ij}^2} D_{f_{ij}}((\theta_0^t)^{(ij)}, (\theta_\star^t)^{(ij)}) \leq \frac{L_{ij}^2}{\mu_{ij}^2} \|(\theta_0^t)^{(ij)} - (\theta_\star^t)^{(ij)}\|^2
\end{align*}
Therefore, 
\begin{equation} \label{eq:warm_start_breg_comp}
    \sum_{i=1}^n \sum_{j=1}^m D_{\phi_{ij}}((\lambda_0^t)^{(ij)}, (\lambda_\star^t)^{(ij)}) \leq \frac{2L_{\max}(C_\omega + 2mC_1)}{\alpha} \varepsilon_{t-1}.
\end{equation}
Substituting Equations~\eqref{eq:warm_start_breg_comm} and~\eqref{eq:warm_start_breg_comp} into Equation~\eqref{eq:warm_start_breg_main} finishes the proof.
\end{proof}

\subsubsection{Inner iteration error decrease}
Now that we have bounded the error at the beginning of each outer iteration, we bound error at the end of each outer iteration by using the convergence results for DVR. We first prove the following Lemma, which controls the distance between the virtual parameters and the actual one:
\begin{lemma}[Virtual error decrease]\label{lemma:virtual_error_after_inner}
For all $(i,j)$, 
\begin{equation}
    \esp{\sum_{i,j}\|(\theta_{t+1})^{(ij)} - \theta_\star^t\|^2} \leq (1 - \rho)^K \left[\|\theta_0^t - \theta_\star^t\|^2_{\comp} + \frac{\rho_{\rm sum} K}{1 - \rho} C_0(t) \right].  
\end{equation}
\end{lemma}

\begin{proof}
We cannot retrieve direct control over the $\theta_{t+1}^{(ij)}$ from control over the dual variables or the dual error, since this would require the $f_{ij}^*$ functions to be smooth, which they may not be. Yet, we leverage the fact that $\theta_{t+1}^{(ij)}$ is obtained by a convex combination between $\theta_{t}^{(ij)}$ and $\theta_{t}^{(i)}$ to obtain convergence of to $\theta_\star^t$. We note $j_{t,k}(i)$ the virtual node that is updated at time $(t,k)$ for node $i$. We note $\mathbb{E}_k$ the expectation relative to the value of $j_{t,k}(i)$. We start by remarking that:
\begin{align*}
    &\espk{k+1}{\|(\theta_{k+1}^t)^{(ij)} - \theta_\star^t\|^2}\\
    &= (1 - p_{ij}) \|(\theta_k^t)^{(ij)} - \theta_\star^t\|^2 + p_{ij} \|(1 - \rho_{ij} )(\theta_k^t)^{(ij)} + \rho_{ij}(\theta_k^t)^{(i)} - \theta_\star^t\|^2\\
    &\leq (1 - p_{ij}\rho_{ij}) \|(\theta_k^t)^{(ij)} - \theta_\star^t\|^2 + p_{ij}\rho_{ij}\|(\theta_k^t)^{(i)} - \theta_\star^t\|^2,
\end{align*}
where in the last inequality we used the convexity of the squared norm. We use that $p_{ij}\rho_{ij} \geq \rho$ (equal for the smallest one), and write that: 
\begin{equation}
    \esp{\|(\theta_{K}^t)^{(ij)} - \theta_\star^t\|^2} \leq (1 - \rho)^K \|(\theta_0^t)^{(ij)} - \theta_\star^t\|^2 + p_{ij}\rho_{ij} \sum_{k=1}^K (1 - \rho)^{k-1} \|(\theta_{K-k}^t)^{(i)} - \theta_\star^t\|^2.  
\end{equation}
Noting $\rho_{\rm sum} = \max_i \sum_{j=1}^m \rho_{ij}p_{ij}$ and $\|\theta_k^t - \theta_\star^t\|^2_{\comp, i} = \sum_{j=1}^m\|(\theta_{k}^t)^{(ij)} - \theta_\star^t\|^2$, we obtain 
\begin{equation}
    \esp{\|\theta_K^t - \theta_\star^t\|^2_{\comp, i}} \leq (1 - \rho)^K \|\theta_0^t - \theta_\star^t\|^2_{\comp, i} + \rho_{\rm sum} \sum_{k=1}^K (1 - \rho)^{k-1} \|(\theta_{K-k}^t)^{(i)} - \theta_\star^t\|^2.  
\end{equation}
Using Lemma~\ref{lemma:primal_error}, we know that $\sum_{i=1}^n \|(\theta_{k}^t)^{(i)} - \theta_\star^t\|^2 \leq C_0(t) (1 - \rho)^k $, with $C_0(t)$ a constant that depends on the initial conditions of outer iteration $t$. Therefore, 
\begin{equation}
    \sum_{i=1}^n\sum_{k=1}^K (1 - \rho)^{k-1} \|(\theta_{K-k}^t)^{(i)} - \theta_\star^t\|^2 \leq K (1 - \rho)^{K - 1}C_0(t).
\end{equation}
In the end, 
\begin{equation}
    \esp{\|\theta_K^t - \theta_\star^t\|^2_{\comp}} \leq (1 - \rho)^K \left[\|\theta_0^t - \theta_\star^t\|^2_{\comp} + \frac{\rho_{\rm sum} K}{1 - \rho} C_0(t) \right].  
\end{equation}
\end{proof}

This lemma has the following corollary:
\begin{corollary}[Warm-started virtual error decrease] \label{corr:inner_iter_theta_comp}
For all $(i,j)$, 
\begin{equation}
    \esp{\sum_{i,j}\|(\theta_{t+1})^{(ij)} - \theta_\star^t\|^2} \leq (1 - \rho)^K \left[6\left(C_\omega + \frac{n}{L}\right) + K\frac{\rho_{\rm sum}C_\comp}{1 - \rho}  \right]\varepsilon_{t-1},
\end{equation}
with $$
C_\comp = \frac{(\beta + \sigma_{\max} + L_{\max})}{(\sigma_{\min} + \beta)^2} \left(\frac{p_{\min}}{\eta_t}C_\phi + C_2 + C_\omega + 4\frac{\beta n}{L}\right)
$$

\end{corollary}

\begin{proof}
    Using Lemmas~\ref{lemma:primal_error},~\ref{lemma:warm_start_breg} and~\ref{lemma:warm_start_dual}, we write: 
    \begin{align*}
        C_0(t) &= \frac{(\beta + \sigma_{\max} + L_{\max})}{(\sigma_{\min} + \beta)^2} \left(\frac{p_{\min}}{\eta_t}D_\phi(\lambda_\star^t, \lambda_0^t) + \left(D(\lambda_\star^t) - D(\lambda_0^t)\right)\right) \leq C_\comp \varepsilon_{t-1}
    \end{align*}
    We use Lemma~\ref{lemma:warm_start_theta_comp} for the first term.
\end{proof}

\begin{lemma}[Condition on $K$]
If Equations~\eqref{eq:recursion_theta_comm},~\eqref{eq:recursion_theta_comp} and~\eqref{eq:recursion_dual_error} hold at time $t$, and  $K$ is such that:
$$ (1 - \rho)^K \leq \min\left(\frac{C_1(1 - \rho^\out)}{12(C_\omega + n/L)}, \frac{C_1(1 - \rho^\out)(1 - \rho)}{K \rho_{\rm sum} C_\comp}, \frac{C_2}{C_L}, \frac{n(\sigma_{\min} + \beta)^2}{2L C_L(\beta + \sigma_{\max} + L_{\max})} \right),$$
then they also hold at time $t+1$.
\end{lemma}

\begin{proof}
Using Corollary~\ref{corr:inner_iter_theta_comp}, we obtain that if $K$ is set such that
$$(1 - \rho)^K \left[6\left(C_\omega + \frac{n}{L}\right) + K\frac{\rho_{\rm sum}C_\comp}{1 - \rho}  \right] \leq C_1(1 - \rho^\out),$$
then the recursion condition is respected for the virtual parameters. This yields the first and second conditions on $K$. Now, we write $C_L = \left(\frac{p_{\min}}{\eta_t}C_\phi + C_D\right)$, then using Lemmas~\ref{lemma:warm_start_breg} and~\ref{lemma:warm_start_dual} (where $C_\phi$ and $C_D$ are defined), we obtain using Theorem~\ref{thm:dvr_dual_guarantees} that
\[D_t(\lambda_\star^{t}) - D_t(\lambda_{t+1}) \leq C_L (1 - \rho)^K \varepsilon_{t-1},\]
since $\lambda_{t+1}$ is obtained by performing $K$ iterations of DVR to minimize $F_t$ starting from $\lambda_t$. This yields the third condition on $K$. Finally, the last condition on $K$ is obtained by leveraging Lemma~\ref{lemma:primal_error}.

\end{proof}

\section{Experiments}
\label{app:experiments}
For the experiments, the following logistic regression problem is solved:
\begin{equation}
    \min_{\theta \in \R^d}  \sum_{i=1}^n \left[\frac{\sigma}{2} \|\theta\|^2 + \sum_{j=1}^m \frac{1}{m}\log(1 + \exp(-y_{ij} X_{ij}^\top \theta)) \right],
\end{equation}
where the pairs $(X_{ij}, y_{ij}) \in \R^d \times \{-1, 1\}$ are taken from the RCV1 dataset, which we downloaded from \url{https://www.csie.ntu.edu.tw/~cjlin/libsvmtools/datasets/binary.html}.

Figure~\ref{fig:plots_full} is the full version of Figure~\ref{fig:plots}, in which we report the number of individual gradients and number of communications for each configuration. We see that accelerated EXTRA actually outperforms EXTRA when the regularization is small, as already mentioned in the main text. We also see that Accelerated EXTRA and Accelerated DVR have comparable communication complexity on the grid graph, when $\gamma$ is smaller. Yet, the computation complexity of (accelerated) DVR is much smaller, so accelerated DVR is much faster overall as long as $\tau$ is not too big. 

\begin{figure}
\subfigure[Erd\H{o}s-R\'enyi, $\sigma= m \cdot 10^{-5}$
]{
    \includegraphics[width=0.32\linewidth]{img/erdos_81_reg5_comp.pdf} \includegraphics[width=0.32\linewidth]{img/erdos_81_reg5_comm.pdf}
    \includegraphics[width=0.32\linewidth]{img/erdos_81_tau_250.pdf}
    \label{fig:erdos_c5_full}
}\\
\subfigure[Grid, $\sigma=m \cdot 10^{-5}$
]{
    \includegraphics[width=0.31\linewidth]{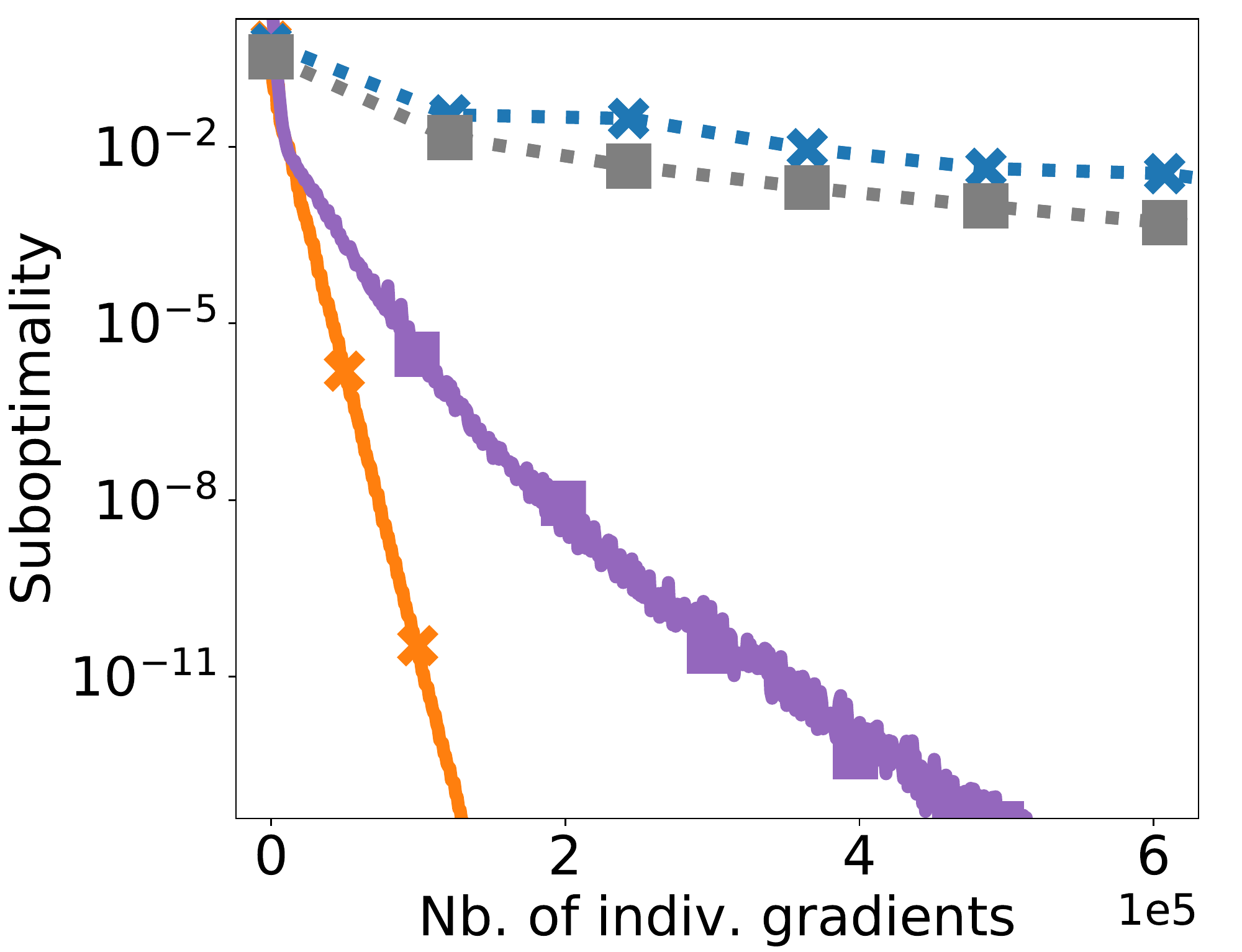}
    \includegraphics[width=0.31\linewidth]{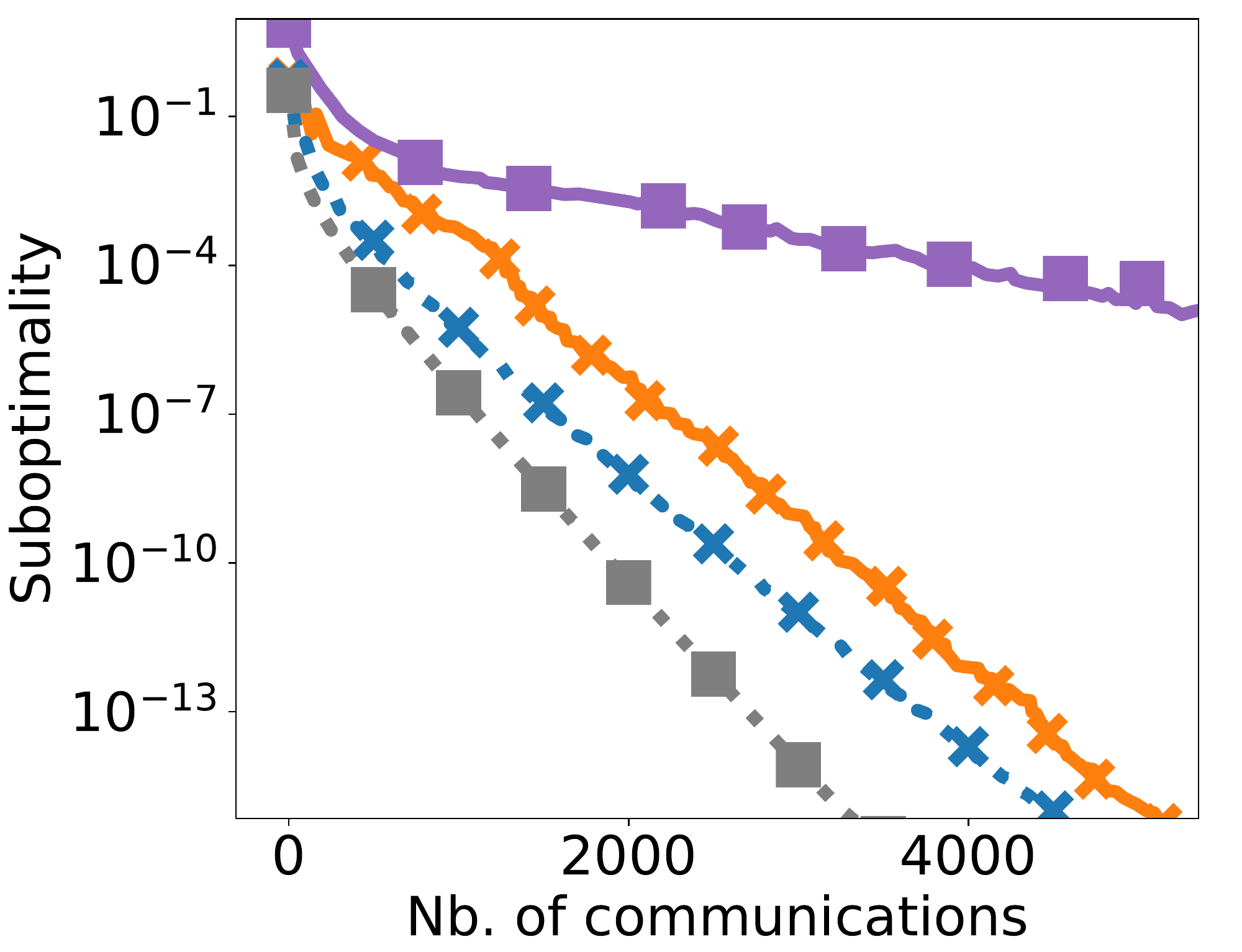}
    \includegraphics[width=0.31\linewidth]{img/grid_81_reg5_tau_250.pdf}
    \label{fig:grid_c5_full}
}
\subfigure[Erd\H{o}s-R\'enyi, $\sigma=m \cdot 10^{-7}$]{
    \includegraphics[width=0.31\linewidth]{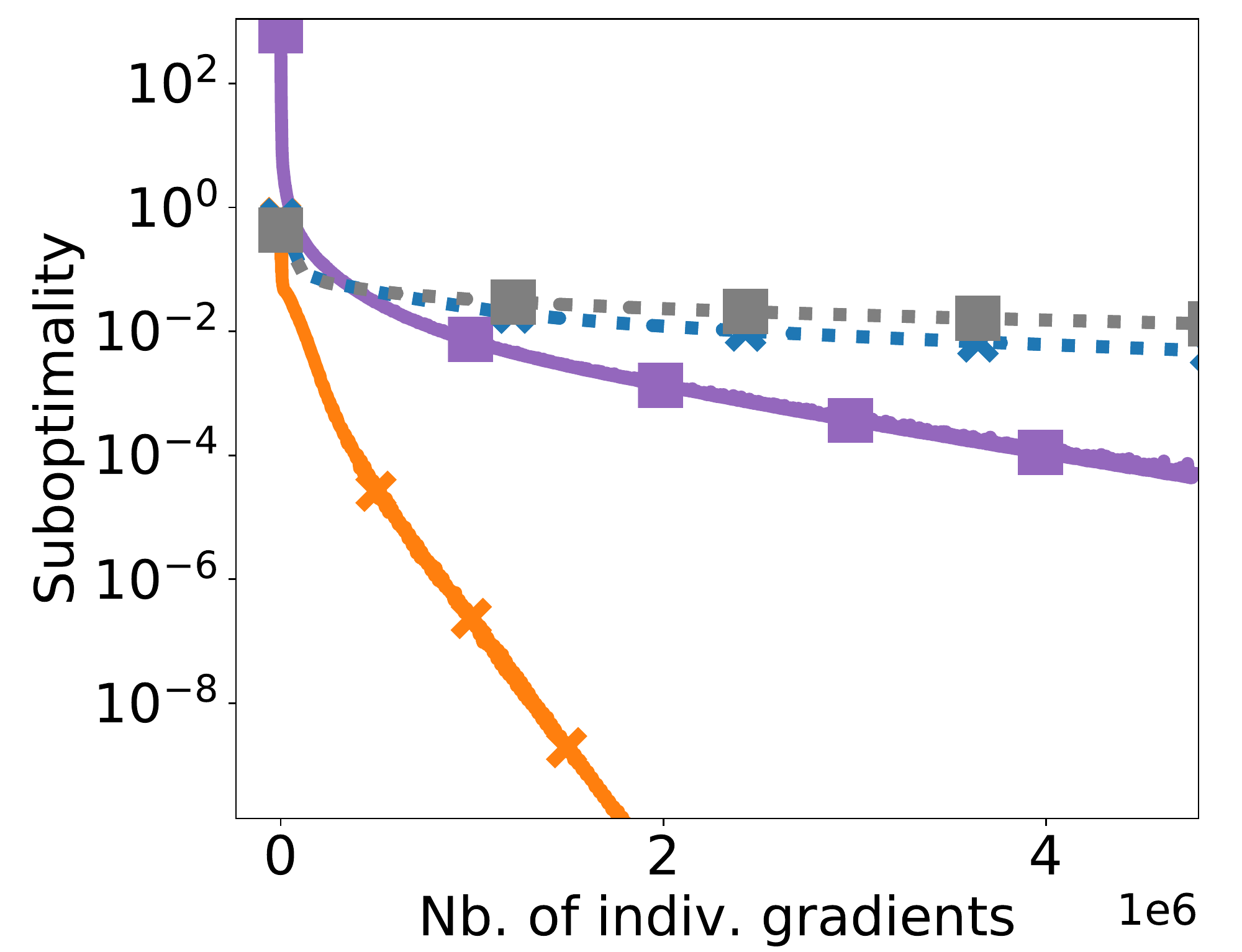}
    \includegraphics[width=0.31\linewidth]{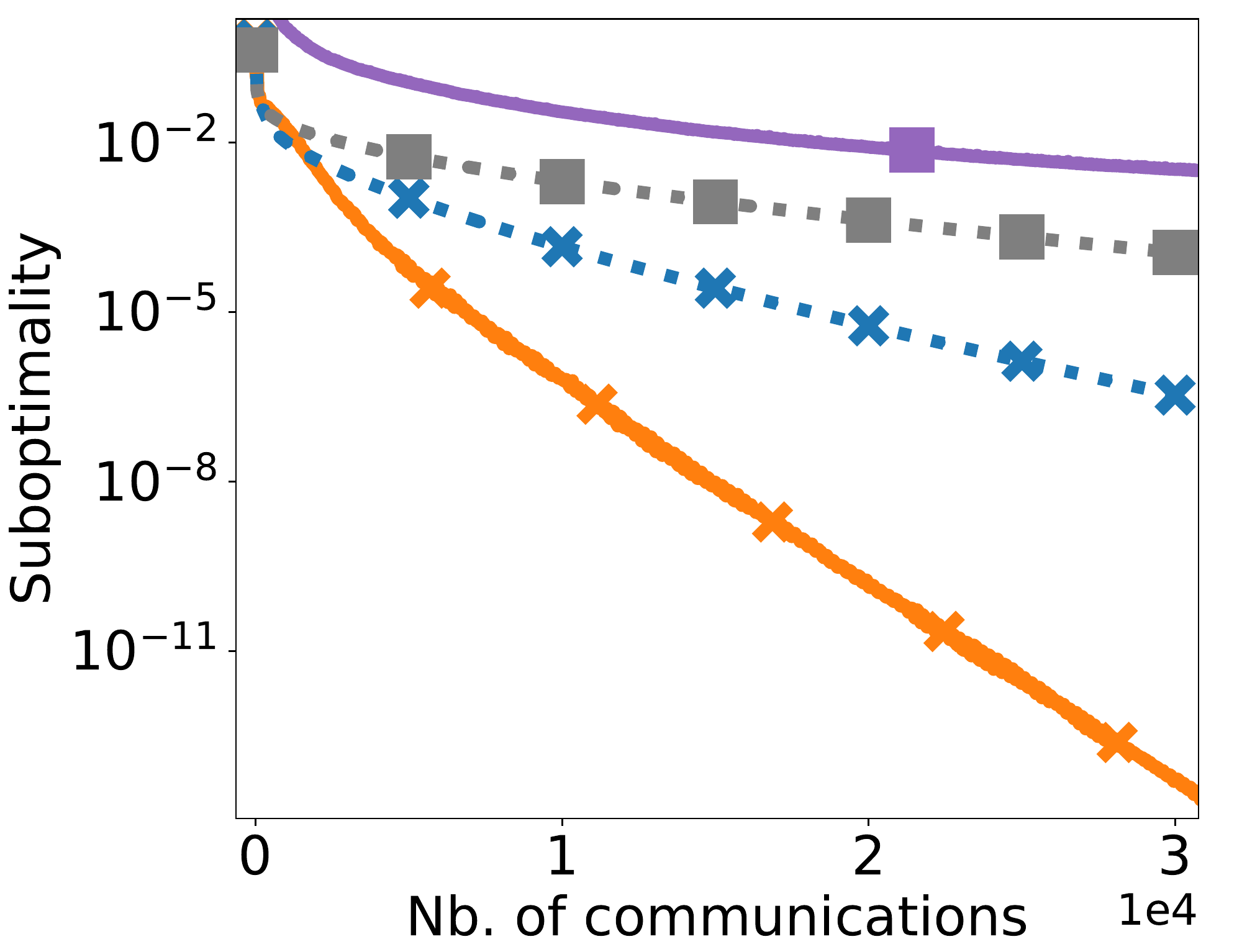}
    \includegraphics[width=0.31\linewidth]{img/erdos_81_reg7_tau_250.pdf}
    \label{fig:grid_c7_full}
}
\subfigure[Grid, $\sigma=m \cdot 10^{-7}$]{
    \includegraphics[width=0.31\linewidth]{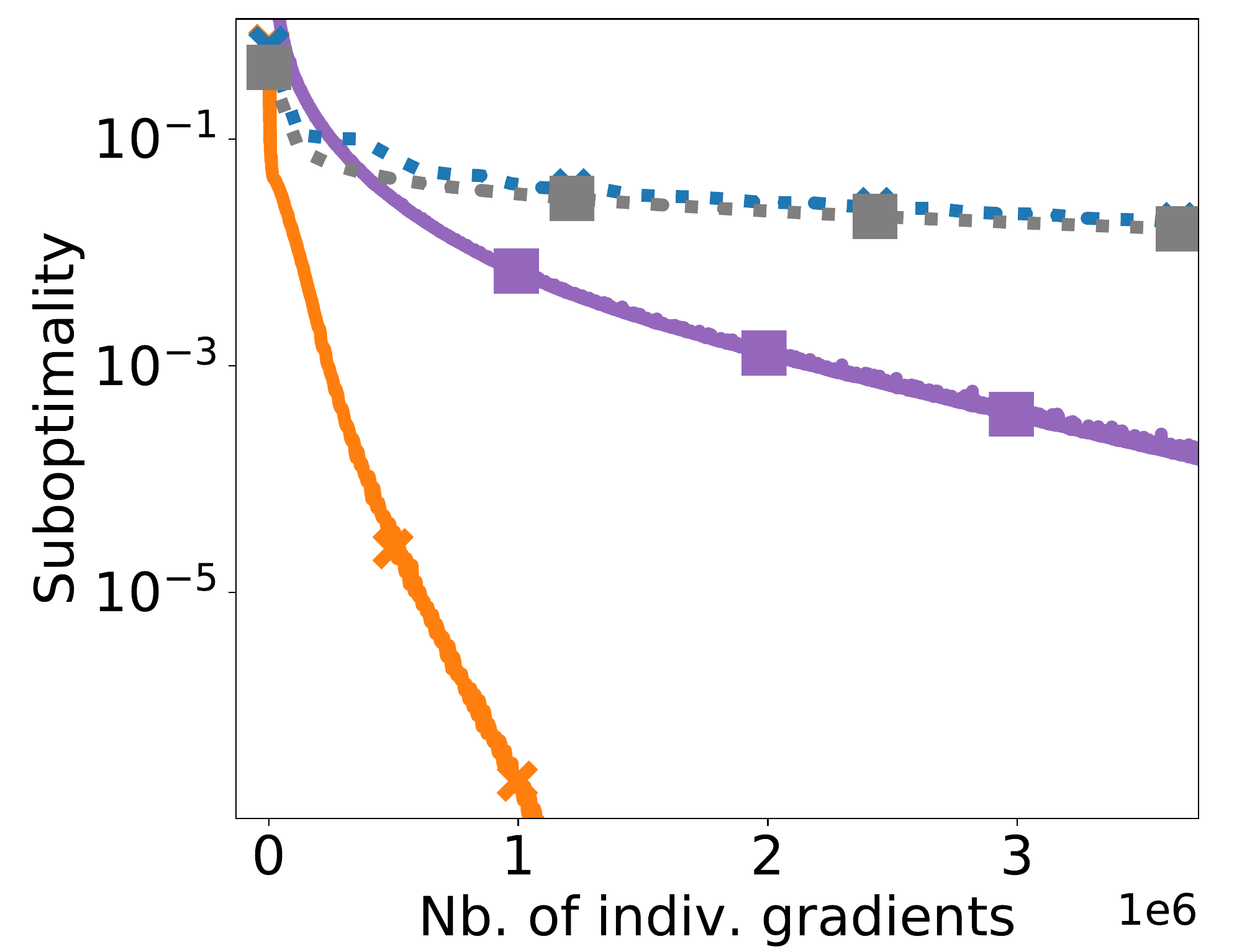}
    \includegraphics[width=0.31\linewidth]{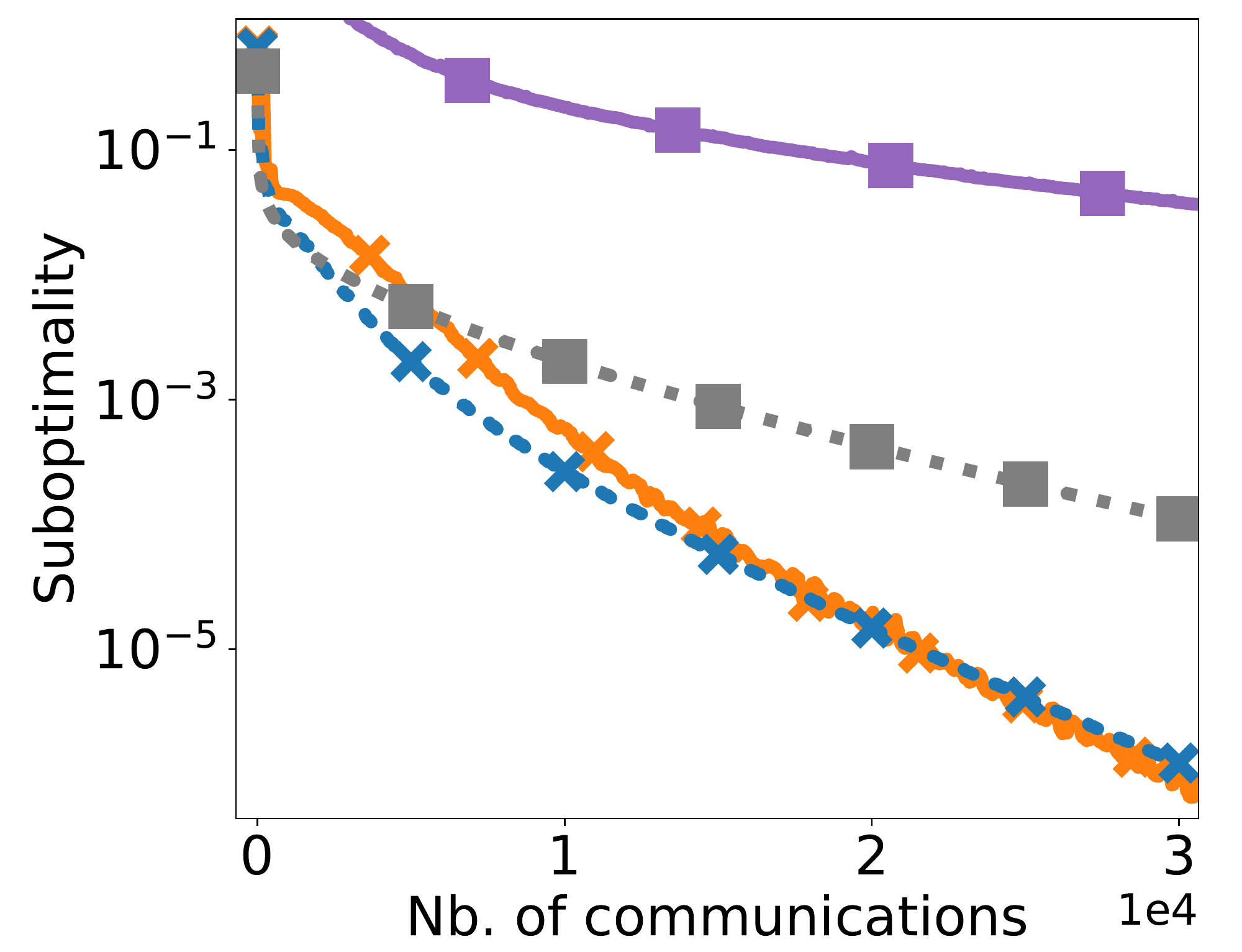}
    \includegraphics[width=0.31\linewidth]{img/grid_81_reg7_tau_250.pdf}
    \label{fig:erdos_c7_full}
}
\caption{Experimental results for the RCV1 dataset with different graphs of size $n=81$, with $m=2430$ samples per node, and with different regularization parameters. \label{fig:plots_full}}
\end{figure}

\end{document}